\documentclass[aos,preprint]{imsart}
\RequirePackage[OT1]{fontenc}
\RequirePackage{amsthm,amsmath}
\RequirePackage[numbers]{natbib}
\usepackage{pictexwd}
\usepackage{bm,amssymb,graphicx,url}
\usepackage{color}
\usepackage{epstopdf}
\usepackage[normalem]{ulem}

\usepackage[colorlinks,citecolor=blue,urlcolor=blue]{hyperref}

\startlocaldefs

\def\a{\alpha}
\def\b{\beta}
\def\th{\theta}
\def\R{\mathbb R}
\def\dd{\Delta}

\def\d{\delta}

\def\E{{\mathbb E}}

\def\IK{I\!\!K}

\def\P{{\mathbb P}}

\def\labda1{\lambda_1}
\def\labda2{\lambda_2}
\def\m{\mu}

\def\e{\varepsilon}
\def\f{\phi}

\def\k{\kappa}

\def\s{\sigma}

\def\comment#1{\relax}

\def\=in{\mathop{\rm =}}

\numberwithin{equation}{section}
\theoremstyle{plain}
\endlocaldefs

\usepackage{amsmath,here,amsfonts,latexsym,graphicx,here}
\usepackage{amsthm,curves}
\usepackage{graphicx}
\usepackage{caption}
\usepackage{subcaption}
\usepackage{float}

\def\m{\mu}
\def\th{\theta}

\def\P{{\mathbb P}}
\def\Q{{\mathbb Q}}

\def\m{\mu}

\def\th{\theta}

\def\P{{\mathbb P}}

\def\dd{\Delta}

\newtheorem{theorem}{Theorem}[section]
\newtheorem{lemma}{Lemma}[section]

\newtheorem{remark}{Remark}[section]
\newtheorem{corollary}{Corollary}[section]
\newtheorem{example}{Example}[section]
\newtheorem{definition}{Definition}[section]

\begin{document}
\begin{frontmatter}
\title{Nonparametric estimation of the incubation time distribution}
\runtitle{incubation time distribution}

\begin{aug}
\author{\fnms{Piet} \snm{Groeneboom}\corref{}\ead[label=e1]{P.Groeneboom@tudelft.nl}
\ead[label=u1,url]{http://dutiosc.twi.tudelft.nl/\textasciitilde pietg/}}
\runauthor{P.\ Groeneboom}
\affiliation{Delft University of Technology}

\address{Delft University of Technology, Mekelweg 4, 2628 CD Delft,
	The Netherlands.\\ 
	\printead{e1}} 
\end{aug}

\begin{abstract}
Nonparametric maximum likelihood estimators (MLEs) in inverse problems often have non-normal limit distributions, like Chernoff's distribution.
However, if one considers smooth functionals of the model, with corresponding functionals of the MLE, one gets normal limit distributions and faster rates of convergence. We demonstrate this for a model for the incubation time of a disease. The usual approach in the latter models is to use parametric distributions, like Weibull and gamma distributions, which leads to inconsistent estimators. Smoothed bootstrap methods are discussed for constructing confidence intervals. The classical bootstrap, based on the nonparametric MLE itself, has been proved to be inconsistent in this situation.
\end{abstract} 



\end{frontmatter}

\section{Introduction}
\label{section:intro}
We consider the following model, used for estimating the distribution of the incubation time of a disease. There is an infection time $U$, uniformly distributed on an interval $[0,E]$, where $E$ (``exposure time'') has an absolutely continuous distribution function $F_E$ on an interval $[0,M_2]$, and where $U$ is uniform on $[0,E]$, conditionally on $E$. Moreover there is an incubation time $V$ with an absolutely continuous distribution $F_0$ on an interval $[0,M_1]$ and a time for getting symptomatic $S$, where $S=U+V$. We assume that $U$ and $V$ are independent, conditionally on $E$.  Our observations consist of the pairs of exposure times and times of getting symptomatic
\begin{align*}
(E_i,S_i),\qquad i=1,\dots,n.
\end{align*}
The model is for example considered in \cite{reich:09}, \cite{tom_gianpi:19}, \cite{backer:20} and \cite{piet:21}.

We define the (convolution) density $q_F$ of $(E_i,S_i)$ by
\begin{align}
\label{convolution}
q_F(e,s)&=e^{-1}\{F(s)-F(s-e)\}\nonumber\\
&=e^{-1}\int_{u=(s-e)_+}^s \,dF(u),\qquad e>0,\,s\in[0,M],
\end{align}
w.r.t.\ $\m$, which is the product of the measure $dF_E$ of the exposure time $E$ and Lebesgue measure on $[0,M]$, where $M=M_1+M_2$ is the upper bound for the time $S$ of getting symptomatic. 
We define the underlying measure $Q_0$ for $(E_i,S_i)$ by
\begin{align}
\label{def_Q_0}
dQ_0(e,s)=q_F(e,s)\,ds\,\,dF_E(e),\qquad s\in[0,M],\qquad e\in(0,M_2].
\end{align}

For estimating the distribution function $F_0$ of the incubation time, usually parametric distributions are used, like the Weibull, log-normal or gamma distribution. However, in \cite{piet:21} the nonparametric maximum likelihood estimator is used. The maximum likelihood estimator $\hat F_n$ maximizes the function
\begin{align}
\label{loglikelihood}
\ell(F)=n^{-1}\sum_{i=1}^n\log\left\{F(S_i)-F(S_i-E_i)\right\}
\end{align}
over {\it all} distribution functions $F$ on $\R$ which satisfy $F(x)=0$, $x\le 0$, see \cite{piet:21}.

Although the model is rather different, the algorithmic problem of computing the MLE has similarities with the problem of computing the MLE in the so-called interval censoring, case 2, model. In the interval censoring, case 2,  model the log likelihood is of the form
\begin{align}
\label{loglikelihood_IC}
\ell(F)&=\sum_{i=1}^n\left\{\dd_{i1}\log F(U_i)+\dd_{i2}\log\{F(V_i)-F(U_i)\}+\dd_{i3}\log\{1-F(V_i)\}\right\},
\end{align}
where $U_i<V_i$ are observation times and we only have information on whether our hidden variable of interest, with distribution function $F$,  is to the left of $U_i$ ($\dd_{i1}=1$), between $U_i$ and $V_i$ ($\dd_{i2}=1$), or to the right of $V_i$ ($\dd_{i3}=1$), see \cite{piet_geurt:14} and \cite{GrWe:92}.

For the incubation time model we have a formally similar way of writing the log likelihood, as can be seen in the following way. First of all, we can introduce, as in \cite{{piet:21}}, the indicator $\dd_{i1}$, defined by
\begin{align}
\label{delta1}
\dd_{i1}=\{S_i\le E_i\}.
\end{align}
If $\dd_{i1}=1$, then $S_i\le E_i$, leading to a term $\log F(S_i)$ in the log likelihood. We can also introduce a second and third indicator $\dd_{i2}$, and $\dd_{i3}$, which depend on whether $S_i\le \max_j(S_j-E_j)$ or $S_i>\max_j(S_j-E_j)$. Note that if $S_i>\max_j(S_j-E_j)$, the distribution function $F$ maximizing (\ref{loglikelihood}) will take the value $1$ at $S_i$, since there is no term $\log\{F(S_j)-F(S_j-E_j)\}$ with $S_j-E_j\ge S_i$. So there is no impediment to giving $F(S_i)$ its maximal value, which is $1$.

Hence, defining this time
\begin{align}
\label{delta2}
\dd_{i2}=\{E_i<S_i\le \max_{j:S_j>E_j}(S_j-E_j)\},
\end{align}
and
\begin{align}
\label{delta3}
\dd_{i3}=\{E_i<S_i,\,S_i>\max_{j:S_j>E_j}(S_j-E_j)\},
\end{align}
we can write the log likelihood in the incubation time model in the form
\begin{align}
\label{loglikelihood_incub}
\ell(F)&=\sum_{i=1}^n\left[\dd_{i1}\log F(S_i)+\dd_{i2}\log\{F(S_i)-F(S_i-E_i)\}+\dd_{i3}\log\{1-F(S_i-E_i)\}\right].
\end{align}
with the $\dd_{ij}$'s defined by (\ref{delta1}) to (\ref{delta3}), using a preliminary reduction of the maximization problem that was also used in \cite{GrWe:92} in the interval censoring, case 2, problem. Note, however, that the indicators $\dd_{i2}$ and $\dd_{i3}$ have a very different meaning in the interval censoring model.

\begin{remark}
\label{remark_M1}
{\rm
Note that if $S_i>M_1$, where $M_1$ is the (unknown) upper bound for the length of the incubation time, we must have:
\begin{align*}
S_i>\max_{j:S_j>E_j}(S_j-E_j)
\end{align*}
since $\max_{j:S_j>E_j}(S_j-E_j)$ is a lower bound for $M_1$. So a distribution function $F$, maximizing the log likelihood, will assign the value $1$ to $F(S_i)$ if $S_i>M_1$.
}
\end{remark}

The limit distribution of the nonparametric MLE of the incubation time distribution was unknown, but because of the (at least algorithmic) similarity of its computation to the computation of the nonparametric MLE in the interval censoring problem, one would expect that similar techniques could be used in its derivation. The algorithmic similarity was indeed used in \cite{piet:21}, but the limit distribution remained an open problem.

On the other hand, the limit distribution of the MLE for the interval censoring problem was derived in \cite{piet:96}, in the so-called strictly separated case, where the length of the observation intervals has a strictly positive lower bound. Unfortunately, the proof is very complicated, and the result seems to be little known. Nevertheless, in the present absence of other methods, this is also the way we proceed in this paper, where we prove the limit result using similar methods, in the hope that this will revive the interest in these matters.

We give the result on the convergence of the rescaled MLE to Chernoff's limit distribution in Section \ref{section:limit distribution}, 
under a condition that seems somewhat similar to the strict separation hypothesis in the interval censoring problem. There is the general expectation that Chernoff's limit distribution will often occur as universal limit distribution in this context of inverse problems, but the difficulty of proving it lies in the fact that we have to derive it from the properties of solutions of (Fredholm) integral equations and that we do not have explicit representations to go on.

As a preparation to this result, we first characterize the MLE as the derivative of the least convex minorant off a self-induced cusum diagram in section \ref{section:char_MLE}. The cusum diagram is an often used tool in the theory of isotonic regression (because the distribution function is monotone, our estimation problem is an isotonic regression problem), but the peculiar feature of the cusum diagrams used here is that they contain the solution $\hat F_n$ (the MLE) itself in their definition. The necessary and sufficient condition in the characterization of the MLE in this way are given in section \ref{section:char_MLE}. We illustrate the (iterative) algorithm for computing the MLE on a data set on COVID-19, also analyzed in \cite{piet:21}. Algorithms of this type (the iterative convex minorant algorithms) were proved to converge in \cite{Jon:98}.

The characterization of section \ref{section:char_MLE} is used to prove consistency of the MLE  in section \ref{section:consistency}. The convergence to Chernoff's distribution is then given in section \ref{section:limit distribution}, where also some numerical results on its variance are given.

Then, in section \ref{section:bootstrap_incubation_time} we define the Smoothed Maximum Likelihood Estimator (SMLE) $\tilde F_{nh}(t)$ at points $t$ away from the boundary by
\begin{equation}
\label{SMLE_incubation}
\tilde F_{nh}(t)=\int \IK_h(t-x)\,d\hat F_n(x),\qquad \IK_h(x)=\IK(x/h),
\end{equation}
where $\IK$ is an integrated kernel, defined by
\begin{equation}
\label{integrated_kernel2}
\IK(x)=\int_{-\infty}^x K(y)\,dy,
\end{equation}
and $K$ is  a symmetric kernel of the usual kind, used in density estimation. Near the boundary we use the Schuster-type boundary correction, also used in Section 11.3 of \cite{piet_geurt:14} in the definition of the SMLE. We assume that $K$ has support $[-1,1]$.

It is shown that the SMLE has a normal limit distribution and  has a faster convergence than the nonparamtric MLE itself (rate $n^{2/5}$ instead of $n^{1/3}$). 
In spite of the fact that its variance is implicitly defined as the solution of an integral equation, we can compute bootstrap confidence intervals for the distribution function via the SMLE, where we do not have to assume that the (asymptotic) variance is known. In this section we also give a method for determining the asymptotically optimal bandwidth automatically (see subsection \ref{subsection:bandwidth_choice}).

It has several advantages to base the confidence intervals on the SMLE instead of the MLE. It follows from recently developed theory (see, e.g., \cite{SenXu2015}) that direct bootstrap confidence intervals, based on the MLE, will be inconsistent. In principle one can base bootstrap confidence intervals on the MLE, using the SMLE intermediately to center the bootstrap intervals, as is done in \cite{SenXu2015} for the interval censoring model. But these intervals converge at a lower speed again and have the unpleasant property of having jumps at the locations of the jumps of the MLE, which is a kind of artefact of the method. Since one needs the SMLE anyway for making the bootstrap confidence intervals consistent, it seems preferable to also base the confidence intervals on the SMLE itself (as is done in the present paper), and not on the MLE. Moreover, the MLE has the property to jump to the value $1$ too early, something that is avoided if one uses the SMLE.

Finally, in sections \ref{section:comparison} to \ref{section:smooth_functionals} we show that the SMLE is competitive with the parametric models, even if the parametric assumptions are satisfied, and avoids the inconsistency that is inherent in the use of the latter methods. Moreover, one is discharged from the duty of introducing several of these parametric methods, since there is no sound reason to choose one of them.

Our computations are reproducible by using the {\tt R} scripts in \cite{github:20}.

\section{Characterization of the nonparametric maximum likelihood estimator (MLE) for the incubation time distribution}
\label{section:char_MLE}
First, just as in \cite{piet:21}, we reduce the problem to the problem of maximizing on the cone $\{\bm y\in\R_+^m:\bm y=(y_1,\dots,y_m),\,0\le y_1\le\dots\le y_m\}$, where the $y_j$ represent the values $F(S_i)$ and $F(S_i-E_i)$ and where $m$ is suitably chosen (see \cite{{piet:21}}).

Since, again just as in \cite{piet:21}, we want to reduce the problem to the problem of maximizing {\it inside} the cone, enabling us to differentiate w.r.t. the variables $y_i$, we define $F(S_i)=1$ if $S_i>\max_{j:S_j>E_j}(S_j-E_j)$, and $F(S_i-E_i)=0$ if $S_i-E_i<\min_jS_n$. Note that other choices of $F$ at these points would make the log likelihood smaller. For convenience of notation, we define
\begin{align}
\label{def_m_n}
m_n=\max_{j:S_j>E_j}(S_j-E_j).
\end{align}

Then, for a distribution function $F$ on $\R_+$, satisfying these conditions, we define the process
\begin{align}
\label{process_W}
W_{n,F}(t)&=\int\frac{\{s\le t\wedge e\}}{F(s)}\,d\Q_n(e,s)-\int\frac{\{0<s-e\le t,\,t<s\le m_n\}}{F(s)-F(s-e)}\,d\Q_n(e,s)\nonumber\\
&\quad+\int\frac{\{s\le t\wedge m_n,\,s>e\}}{F(s)-F(s-e)}\,d\Q_n(e,s)-\int\frac{\{s-e\le t,\,s>e\vee m_n\}}{1-F(s-e)}\,d\Q_n(e,s).
\end{align}
where we define $0/0=0$, and where $\Q_n$ is the empirical distribution of $(E_1,S_1),\dots,(E_n,S_n)$. Note that the MLE is only determined at the points $S_i$ and $(S_i-E_i)1_{\{S_i>E_i\}}$ and is zero at $0$.

We restrict the distribution functions, occurring in the problem of maximizing the likelihood to the following set.

\begin{definition}
\label{def_cal_F_n}
{\rm
Let ${\cal F}_n$ be the set of discrete distribution functions $F$, which only have mass at the points $S_i$ or $(S_i-E_i)1_{\{S_i>E_i\}}$ and satisfy 
\begin{align}
\label{upper_values}
F(S_i)=1\quad \text{\rm if }\quad S_i>\max_{j:S_j>E_j}(S_j-E_j),
\end{align}
and
\begin{align}
\label{lower_values}
F(S_i-E_i)=0\quad \text{\rm if }\quad S_i-E_i<\min_jS_j.
\end{align}
}
\end{definition}

\vspace{1cm}
Now, let $T_1< \dots < T_m$ be the points $S_i$ or $S_i-E_i$ such that $S_i$ not of type (\ref{upper_values}) and $S_i-E_i$ is not of type (\ref{lower_values}). Then $0<F(T_i)<1$ for $i=1,\dots,m$, if the log likelihood for $F$ is finite, and we can define:
\begin{align*}
{\cal Y}=\{\bm y\in(0,1)^m:\bm y=(y_1,\dots,y_m)=(F(T_1),\dots,F(T_m),\,F\in{\cal F}_n\}.
\end{align*}
The log likelihood for $F\in{\cal F}_n$, divided by $n$, can be written
\begin{align}
\label{emp_loglik}
\ell_n(F)&=\int\log\{F(s)-F(s-e)\}\,d\Q_n(s,e)\nonumber\\
&=\int_{s\le e}\log F(s)\,d\Q_n(e,s)+\int_{e<s\le m_n}\log\{F(s)-F(s-e)\}\,d\Q_n(e,s)\nonumber\\
&\quad\qquad\quad\qquad\quad\qquad\quad\qquad+\int_{s>e\vee m_n}\log\{1-F(s-e)\}\,d\Q_n(e,s).
\end{align}
The corresponding function of $\bm y=(y_1,\dots,y_m)=(F(T_1),\dots,F(T_m))$ is denoted by $\f_n(\bm y)$:
\begin{align}
\label{phi_n}
\f_n(\bm y)=\f_n(y_1,\dots,y_m)=\f_n(F(T_1),\dots,F(T_m))=\ell_n(F).
\end{align}

We have the following lemma.

\begin{lemma}
Let $W_{n,F}$ be defined by (\ref{process_W}) and let $\f_n$ be defined by (\ref{phi_n}). Moreover, let
\begin{align*}
\bm y=(F(T_1,\dots,F(T_m)),
\end{align*}
where $T_1< \dots < T_m$ are the points $S_i$ and $(S_i-E_i)1_{\{S_i>E_i\}}$ which are not of type (\ref{upper_values}) or (\ref{lower_values}), arranged in strictly increasing order 
Then
\begin{align}
\label{representation_partial_der}
\frac{\partial}{\partial y_j}\f_n(\bm y)=\dd W_{n,F}(T_j),\qquad j=1,\dots,m,
\end{align}
where $\dd W_{n,F}(T_j)$ is the increment of the process $W_{n,F}$ at $T_j$.
\end{lemma}
 
\begin{proof}
If $T_j$ corresponds to a value $S_i$ such that $S_i\le E_i$, the corresponding term in (\ref{emp_loglik}) is of the form
$\log F(S_i)$ and differentiation of (\ref{phi_n}) w.r.t. $F(S_i)$ gives the following contribution to the partial derivative $\frac{\partial}{\partial y_j}\f_n(\bm y)$:
$$
\sum_k\frac{\{S_k\le E_k,\,S_k=S_i\}}{nF(S_i)},
$$
where we make a summation over the $k$ to allow for possible ties at $S_i$.

If $T_j$ corresponds to a point $S_i-E_i$ such that $T_j=S_i-E_i>0$ and $T_j<S_i\le m_n$, we deal with a term
\begin{align*}
\log\{F(S_i)-F(S_i-E_i)\}
\end{align*}
in (\ref{emp_loglik}) and differentiation of $\f_n(\bm y)$ w.r.t. $F(S_i-E_i)$ gives a contribution
$$
-\sum_k\frac{\{T_j\vee E_k<S_k\le m_n,\,S_k-E_k=S_i-E_i\}}{n\{F(S_k)-F(S_k-E_k)\}},
$$
where we make again a summation over $k$ to allow for possible ties at $S_i$.

If $T_j$ corresponds to a point $S_i$ such that $T_j=S_i>E_i$ and $S_i\le m_n$, we deal with the argument $F(S_i)$ of the term
\begin{align*}
\log\{F(S_i)-F(S_i-E_i)\}
\end{align*}
in (\ref{emp_loglik}) and differentiation w.r.t. $F(S_i)$ gives a term
$$
\sum_k\frac{\{E_k<S_k\le m_n,\,S_k=S_i\}}{n\{F(S_k)-F(S_k-E_k)\}},
$$
where we make again a summation over $k$ to allow for possible ties at $S_i$.

Finally, if $T_j$ corresponds to a point $S_i-E_i$ such that $S_i>m_n\vee E_i$ , we deal with a term
\begin{align*}
\log\{1-F(S_i-E_i)\}
\end{align*}
and differentiation and differentiation of $\f_n(\bm y)$ w.r.t. $F(S_i-E_i)$ gives a contribution
$$
-\sum_k\frac{\{S_k> m_n,\,S_k-E_k=S_i-E_i\}}{n\{1-F(S_k-E_k)\}},
$$
So we get:
\begin{align*}
\frac{\partial}{\partial y_j}\f_n(\bm y)=\dd W_{n,F}(T_j),\qquad j=1,\dots,m.
\end{align*}
\end{proof}

The following lemma characterizes the MLE.
\begin{lemma}
\label{lemma:char_MLE}
Let the class of distribution functions ${\cal F}_n$ be defined by Definition \ref{def_cal_F_n}. Then $\hat F_n\in{\cal F}_n$ maximizes (\ref{loglikelihood}) over $F\in{\cal F}_n$ if and only if
\begin{enumerate}
\item[(i)] 
\begin{align}
\label{fenchel1}
\int_{u\in[t,\infty)}\,dW_{n,\hat F_n}(u)\le0,\qquad t\ge0,
\end{align}
\item[(ii)]
\begin{align}
\label{fenchel2}
\int \hat F_n(t)\,dW_{n,\hat F_n}(t)=0.
\end{align}
\end{enumerate}
where $W_{n,F_n}$ is defined by (\ref{process_W}). Moreover, $\hat F_n\in{\cal F}_n$ is uniquely determined by (\ref{fenchel1}) and (\ref{fenchel2}). 
\end{lemma}

\begin{proof}
Suppose (\ref{fenchel1}) and (\ref{fenchel2}) are satisfied. Letting $\bm y=(F(T_1,\dots,F(T_m))$ and $\bm x=(\hat F_n(T_1),\dots,\hat F_n(T_m))$, we get from the concavity of the logarithmic function:
\begin{align*}
&\ell_n(F)-\ell_n(\hat F_n)\le \langle \nabla\f_n(\bm x),\bm y-\bm x\rangle\\
&=\int\{F(t)-\hat F_n(t)\}\,dW_{n\hat F_n}(t)=\int F(t)\,dW_{n\hat F_n}(t),
\end{align*}
using (ii) in the last step. But
\begin{align*}
\int F(t)\,dW_{n\hat F_n}(t)=\int_{0\le u\le t}\,dF(u)\,dW_{n\hat F_n}(t)
=\int\left\{\int_{t\in[u,\infty)}\,dW_{n\hat F_n}(t)\right\}\,dF(t)\le0,
\end{align*}
using (i) in the last inequality.

Conversely, if $\hat F_n$ maximizes the likelihood and
\begin{align*}
\bm y=(\hat F_n(T_1),\dots,\hat F_n(T_m)),
\end{align*}
we have
\begin{align*}
\lim_{\e\downarrow0}\e^{-1}\left\{\f_n(y_1,\dots,y_i+\e,\dots,y_m+\e)-\f_n(y_1,\dots,y_m)\right\}
=\int_{u\in[T_i,\infty)}\,dW_{n,\hat F_n}(u)\le0,
\end{align*}
for $i=1,\dots,m$, and hence
\begin{align*}
\int_{u\in[t,\infty)}\,dW_{n,\hat F_n}(u)\le0,\qquad t\ge0,
\end{align*}
and similarly
\begin{align*}
\lim_{\e\downarrow0}\e^{-1}\left\{\f_n(y_1+\e y_1,\dots,y_m+\e y_m)-\f_n(y_1,\dots,y_m)\right\}=\int \hat F_n(t)\,dW_{n,\hat F_n}(t)=0.
\end{align*}
The uniqueness can be proved aong the same lines as in the proof of Proposition 1.3 in \cite{GrWe:92}. We omit the details.
\end{proof}

\begin{figure}[!ht]
\centering
\includegraphics[width=0.5\textwidth]{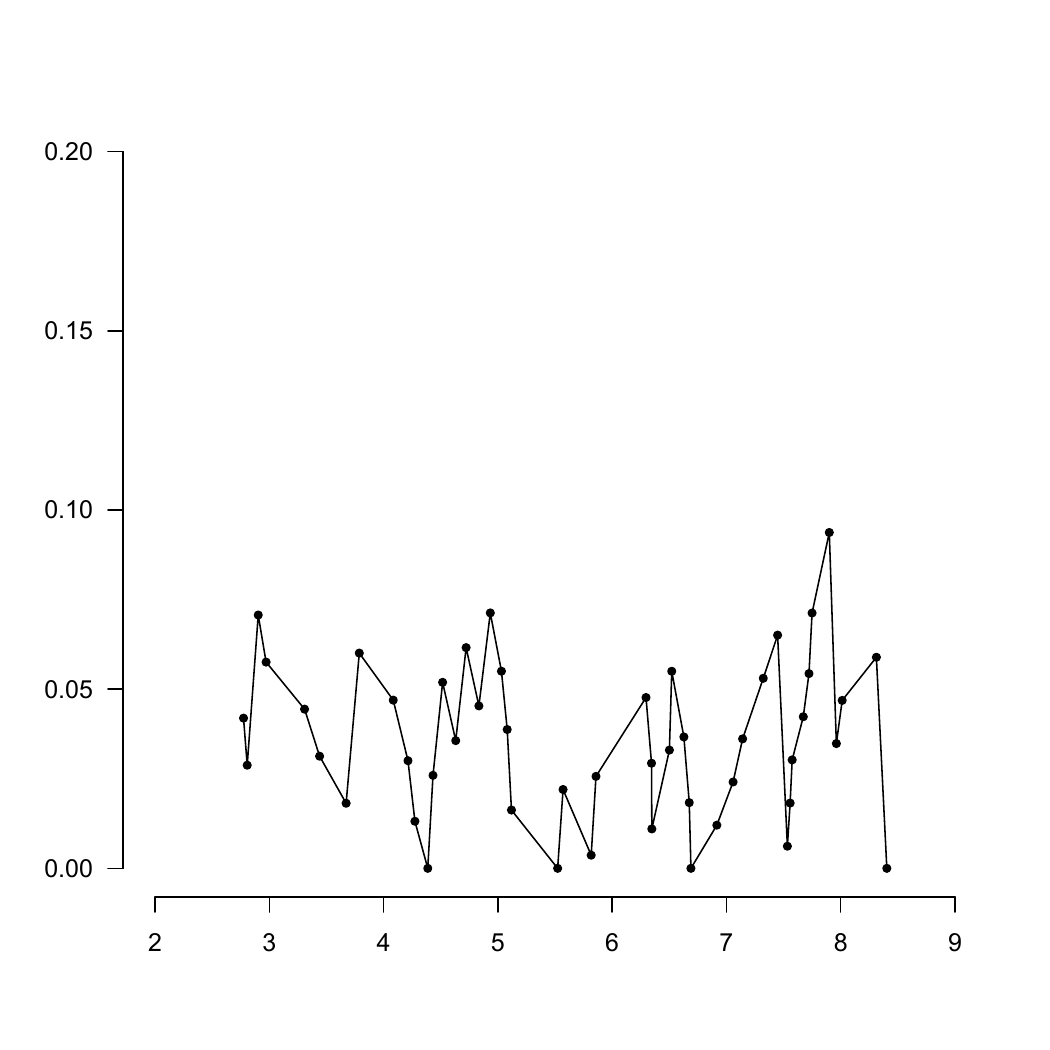}
\caption{The point process $\{(T_i,W_{n,\hat F_n}(T_i)),\,i=1,2,\dots\}$ for points $T_i$, running through points $S_i$ and $S_i-E_i$ between the minimum of the $S_i$ and the maximum of the $S_i-E_i$. Sample size $n=100$. The data correspond to a truncated Weibull distribution for the incubation time distribution, used in simulations of the incubation time distribution. The points are connected by line segments.}
\label{figure:W_process}
\end{figure}

A picture of the point process $\{(T_i,W_{n,\hat F_n}(T_i)),\,i=1,2,\dots\}$ in a simulation of the incubation time distribution, for $T_i$ running through the points $S_i$ and $(S_i-E_i)_+$ between the minimum of the $S_i$ and the maximum of the $(S_i-E_i)_+$. is given in Figure \ref{figure:W_process} for sample size $n=100$.
The process $W_{n,\hat F_n}$  touches zero at points just to the left of points of mass of $\hat F_n$.

We define the process $V_n$ by
\begin{align}
\label{def_V_n}
V_n(t)=\int_{u\in[0,t]}\hat F_n(u)\,dG_n(u)+W_{n,\hat F_n}(t),
\end{align}
where $G_n = G_{n,\hat F_n}$ is defined by (\ref{G_{n,F}}) for $F=\hat F_n$. Thus $\hat F_n$ is obtained by taking the left-continuous slope of the ``self-induced'' cusum diagram, defined by $(0,0)$ and points
\begin{align}
\label{cusum2}
\left(G_n(t),V_n(t)\right),\qquad t\ge0.
\end{align}

As explained in \cite{piet:21}, one can compute the MLE by the iterative convex minorant algorithm, where one computes iteratively the greatest convex minorant of the cusum diagram with points $(0,0)$ and points
\begin{align}
\label{cusum}
\left(G_{n,F}(t),\int_{u\in[0,t]} F(u)\,dG_{n,F}(u)+W_{n,F}(t)\right),
\end{align}
where the ``weight process'' $G_{n,F}$ is defined by
\begin{align}
\label{G_{n,F}}
G_{n,F}(t)&=\int_{s\le t}\frac{\{s\le t\wedge e\}}{F(s)^2}\,d\Q_n(e,s)+\int_{0<s-e\le t}\frac{\{0<s-e\le t,\,t<s\le m_n\}}{\{F(s)-F(s-e)\}^2}\,d\Q_n(e,s)\nonumber\\
&\qquad\qquad\qquad\qquad+\int_{s\vee e\le t}\frac{\{s\le t\wedge m_n,\,s>e\}}{\{F(s)-F(s-e)\}^2}\,d\Q_n(e,s)\nonumber\\
&\qquad\qquad\qquad\qquad\qquad\qquad+\int_{s-e\le t,\,s>m_n}\frac{\{s-e\le t,\,s>e\vee m_n\}}{\{1-F(s-e)\}^2}\,d\Q_n(e,s),
\end{align}
where $F$ is the temporary estimate of the distribution function at an iteration. The MLE $\hat F_n$ corresponds to a stationary point of this algorithm and is given by the left-continuous slope of the greatest convex minorant of the cusum diagram, see Figure \ref{figure:cusum}. See \cite{piet:21} for further  remarks on this algorithm. The algorithm is implemented in the {\tt R} scripts in \cite{github:20}.

\begin{figure}[!ht]
\centering
\includegraphics[width=0.5\textwidth]{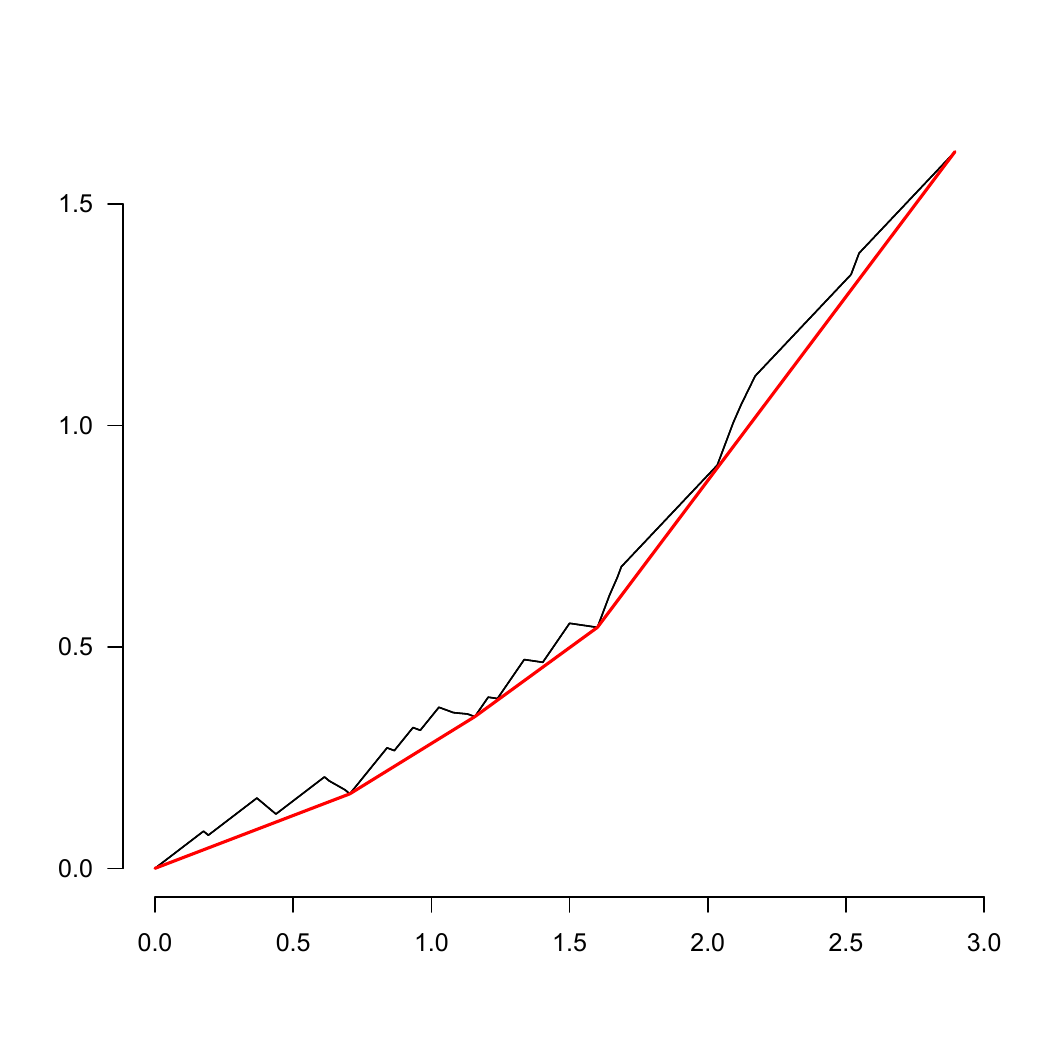}
\caption{The cusum diagram $\{(G_n(T_i),V_n(T_i)),\,i=1,2,\dots\}$, where $T_i$ runs through the ordered points $(S_i-E_i)_+$ and $S_i$ and $V_n$ is defined by (\ref{def_V_n}), together with its greatest convex minorant (red curve). Sample size $n=100$.}
\label{figure:cusum}
\end{figure}

The steps in the iterative algorithm are modified by a line search algorithm, see section 7.3 of \cite{piet_geurt:14} and in particular the description of this modified iterative convex minorant algorithm on p.\ 173. The modified iterative convex minorant algorithm is guaranteed to converge by Theorem 7.3 of \cite{piet_geurt:14}, see also \cite{Jon:98}. The modified version is used in the {\tt R} scripts acccompanying this paper \cite{github:20}.

\begin{example}
\label{example_Wuhan}
{\rm
The first time the methods just described were used  for the estimation of the incubation time distribution for Covid-19 was in the analysis of data on 88 travelers from Wuhan  in \cite{piet:21}. The data set is included in \cite{piet:21} and extracted from the supplementary material of \cite{backer:20}. In this case we have:
\begin{align*}
\f_n(\bm y)=\sum_{0<j\le m}N_{0i}\log y_j+\sum_{0<i<j\le m}N_{ij}\log(y_j-y_i)+\sum_{i\le m}N_{i,m+1}\log(1-y_i),
\end{align*}
where $m=6$, $\f_n$ is defined by (\ref{emp_loglik}), and the matrix $(N_{ij})$, $0\le i<j\le m+1$ is given by:
\begin{align*}
\begin{array}{llllllll}
&1 \qquad &3  \qquad &4 \qquad  &0 \qquad  &0  \qquad &2 \qquad  &0 \\  
&  &2    &1   &0   &0   &0   &9\\   
&&&0   &1   &1   &0   &4\\   
&&&& 1    &0   &2   &3  \\
&&&&&1     &0   &6 \\ 
&&&&&& 1     &3   \\
&&&&&&& 3 
\end{array}
\end{align*}
The corresponding points $T_1,\dots,T_6$ are given by $T_i=i+2$ and the MLE is given in Figure \ref{figure:MLE_Wuhan_data}. 
For  comparison the maximum likelihood estimator, assuming that the incubation time distribution is a Weibull distribution, is also given in this picture. The nonparametric MLE $\hat F_n$ was denoted by $\hat G_n$ in \cite{piet:21}.

The result can be reproduced by running the {\tt R} script {\tt analysis{\_}ICM.R} in \cite{github:20}.

\begin{figure}[!ht]
\centering
\includegraphics[width=0.5\textwidth]{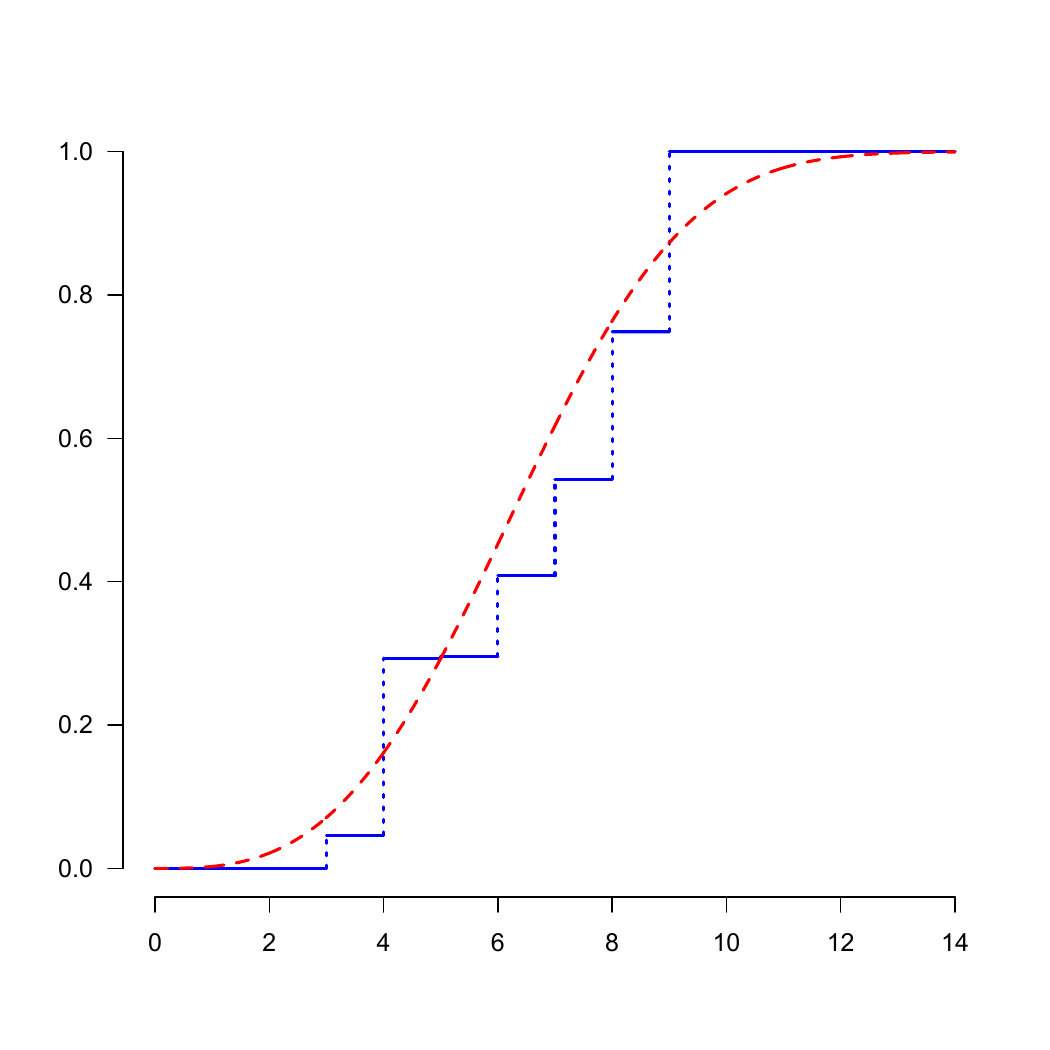}
\caption{The nonparametric MLE $\hat F_n$ of the incubation time distribution function (blue), and the MLE using the Weibull distribution (red, dashed), for the $88$ Wuhan travelers.}
\label{figure:MLE_Wuhan_data}
\end{figure}

}
\end{example}

\section{Consistency of the MLE}
\label{section:consistency}
We have the following result.

\begin{theorem}
\label{th:consistency}
Let the incubation time distribution function $F_0$ have a strictly positive continuous density $f_0$ on $(0,M_1)$, for some $M_1>0$. Furthermore, let $F_E$ be a distribution function on $[0,M_2]$, $M_2>M_1/2$, which is $0$ on $[0,\e]$ for some $\e\in [0,(M_1\wedge M_2)/2)$ and has a continuous strictly positive derivative $f_E$ on $[\e,M_2]$. Let $\hat F_n\in{\cal F}_n$ be the nonparametric MLE, where ${\cal F}_n$ is the set of distribution functions defined in Definition \ref{def_cal_F_n}. Then $\hat F_n$ converges almost surely to $F_0$ on $[0,M_1]$ in the supremum metric.
\end{theorem}
There are a lot of different ways to prove consistency, but we feel a preference for the elegant method in \cite{jewell:82}, which is used in the proof below.

\begin{proof}
We use the same method as in section 4.2 of the second part of the book \cite{GrWe:92}. Let $\psi(F)$ be defined by
\begin{align*}
\psi(F)=\int \log\{F(s)-F(s-e)\}\,d\Q_n(e,s),
\end{align*}
for distribution functions $F$ on $\R$, which satisfy $F(x)=0$, $x\le0$. Then we get:
\begin{align*}
\lim_{\e\downarrow0}\e^{-1}\left\{\psi\left((1-\e)\hat F_n+\e F_0\right)-\psi\left(\hat F_n\right)\right\}\le0,
\end{align*}
since $\hat F_n$ is the MLE. The limit exists because of the concavity of $\psi$. Evaluating this limit, we get:
\begin{align*}
\int\frac{F_0(s)-F_0(s-e)}{\hat F_n(s)-\hat F_n(s-e)}\,d\Q_n(e,s)\le1.
\end{align*}

Proceeding as in section 4.2 of the second part of the book \cite{GrWe:92}, we get from this,
if $F$ is a limit point for $\hat F_n$ (using Helly's compactness theorem):
\begin{align}
\label{limit_ineq}
\int e^{-1}\frac{\{F_0(s)-F_0(s-e)\}^2}{F(s)-F(s-e)}\,ds\,dF_E(e)\le1.
\end{align}
We want to show that the minimum over $F$ on the left is equal to $1$, and that this can only be attained if $F=F_0$.  Note that by Remark \ref{remark_M1}, $\hat F_n(s)=1$, if $s\ge M_1$, and hence also $F(s)=1$, if $s\ge M_1$. This implies
\begin{align}
\label{F_nondegenerate}
\int e^{-1}\{F(s)-F(s-e)\}\,ds\,dF_E(e)=1.
\end{align}
We therefore have:
\begin{align}
\label{criterion_F}
&\int e^{-1}\left\{\frac{\{F_0(s)-F_0(s-e)\}^2}{F(s)-F(s-e)}+\{F(s)-F(s-e)\right\}\,ds\,dF_E(e)\nonumber\\
&=\int_{s\le e} e^{-1}\left\{\frac{F_0(s)^2}{F(s)}+F(s)\right\}\,ds\,dF_E(e)\nonumber\\
&\qquad+\int_{e<s\le M_1} e^{-1}\left\{\frac{\{F_0(s)-F_0(s-e)\}^2}{F(s)-F(s-e)}+\{F(s)-F(s-e)\right\}\,ds\,dF_E(e)\nonumber\\
&\qquad+\int_{s>M_1\vee e} e^{-1}\left\{\frac{\{1-F_0(s-e)\}^2}{1-F(s-e)}+1-F(s-e)\right\}\,ds\,dF_E(e)\nonumber\\
&\le2.
\end{align}

The function
\begin{align*}
(x,y)\mapsto \frac{\{F_0(s)-F_0(s-e)\}^2}{y-x}+y-x
\end{align*}
is minimized if $x=F_0(s-e)$ and $y=F_0(s)$ (also if $F_(s-e)=0$ or $F_0(s)=1$). If $F$ would not be equal to $F_0$, it would be different from $F_0$ on an interval, using the monotonicity of $F$ and $F_0$ and the continuity of $F_0$, and then the left side of (\ref{criterion_F}) would be strictly larger than $2$. Hence, by (\ref{F_nondegenerate}), the left side of (\ref{limit_ineq}) would be strictly larger than $1$, a contradiction.
\end{proof}

\section{Asymptotic distribution of the MLE in the model for the incubation time}
\label{section:limit distribution}

We have the following result for the MLE in the model for the incubation time.

\begin{theorem}
\label{th:local_limit}
Let the conditions of Theorem \ref{th:consistency} be satisfied and let, moreover, $f_E$ have a bounded derivative on the interval $(\e,M_2)$. Let $\hat F_n\in{\cal F}_n$ be the nonparametric MLE, where the set of distribution functions ${\cal F}_n$ is defined in Definition \ref{def_cal_F_n}, and let $F_0$ be the distribution function of the incubation time.  Then we have at a point $t_0\in(0,M_1)$:
\begin{align}
\label{local_limit_result}
n^{1/3}\{\hat F_n(t_0)-F_0(t_0)\}/(4f_0(t_0)/c_E)^{1/3}\stackrel{d}\longrightarrow \text{\rm argmin}\left\{W(t)+t^2\right\},
\end{align}
where $W$ is two-sided Brownian motion on $\R$, originating from zero and where the constant $c_E$ is given by:
\begin{align}
\label{c_E}
c_E=\int e^{-1}\left[\frac1{F_0(t_0)-F_0(t_0-e)}
+\frac1{F_0(t_0+e)-F_0(t_0)}\right]\,dF_E(e),
\end{align}
\end{theorem}

The result shows that the limit distribution is given by Chernoff's distribution. 
The jump in difficulty of the proof in going from the corresponding result for the current status model (not discussed in this paper) to more general cases of  interval censoring models and to the model for the incubation time distribution is considerable. One expects in fact that  Chernoff's distribution will often occur as (a universal) limit distribution in these contexts, but proving this might be very hard.

The proof of Theorem \ref{th:local_limit} is given in the Appendix, section \ref{subsection:proof_theorem4}. For the interval censoring, case 2, model the limit distribution of the MLE, under the so-called strict separation condition, was derived in \cite{piet:96}. The strict separation condition in the interval censoring, case 2, model seems somewhat comparable to the condition that $E_i$ has no mass on an interval $[0,\e]$ in the present model. That the exposure time (as observed!) has a strictly positive lower bound does not seem such an unreasonable assumption.

At several places of the proof in section \ref{subsection:proof_theorem4} of the appendix, we use arguments improving on the arguments in \cite{piet:96}, and the proof of the limit result for the interval censoring model in \cite{piet:96} could be improved similarly. But we will not go into these matters in this paper.

It is of interest to investigate whether the nonparametric MLE has sample variances that resemble the asymptotic variance of Theorem \ref{th:local_limit}. To this end we computed the variances over $1000$ simulations of the MLE $\hat F_n(t)$ at $t=6$, if the underlying distribution function is the truncated Weibull distribution function $F_{\a,\b}$, defined by
\begin{align}
\label{weibull_df}
F_{\a,\b}(x)\stackrel{\text{def}}=\left\{\begin{array}{ll}
0 &,\,x<0,\\
\left\{1-\exp(-\b x^{-\a})\right\}/\left\{1-\exp(-M_1 x^{-\a})\right\}\,\qquad &,\,x\in[0,M_1],\\
1 &,\,x>M_1,
\end{array}
\right.
\end{align}
where we choose $M_1=20$. We chose $\a=3.03514$ and $\b=0.0026195$, which were the values of the estimates in the data on the Wuhan travelers, discussed in \cite{piet:21} and section \ref{section:bootstrap_incubation_time} of the present paper, if one assumes that $F_0$ is a truncated Weibull distribution. The distribution function $F_E$ of the exposure time was taken to be the uniform distribution on the interval $[1,30]$.

\begin{table}[!ht]
\caption{$n^{2/3}$ times the variances of $\hat F_n(6)$ for $1000$ simulations for the model, where $F_E$ is uniform on $[1,30]$ and $F_0$ is a truncated Weibull distribution on $[0,20]$. The limit value of Theorem \ref{th:local_limit} is denoted. by $\infty$.}
	\centering
	\label{table:simulation1b}
		\begin{tabular}{|c|c|c|c|c|c|c|c|c|}
			\hline
			 $n$ & $100$ & $500$ & $1000$ & $5000$ & $10,000$ & $\infty$\\
			\hline
			\text{$n^{2/3}\cdot$ variance} & $0.38990$ & $0.36071$ & $0.32329$ & $0.28816$  & $0.27188$ & $0.27489$\\
			\hline	
		\end{tabular}
\end{table}

The asymptotic variance, given by Theorem \ref{th:local_limit}, was computed using the software package Mathematica, where the asymptotic variance $\s^2$ of the location of the minimum of $W(t)+t^2,\,t\in\R$, was taken from \cite{piet_jon:01},
see Table 4 of \cite{piet_jon:01}, where the value $\s^2=0.26355964$ is given.

The results suggest a downward trend of the sample variances to the asymptotic value.

\section{Confidence intervals for the distribution function}
\label{section:bootstrap_incubation_time}
We now construct pointwise confidence intervals for the distribution function on the basis of the SMLE (smoothed maximum likelihood estimator), defined by (\ref{SMLE_incubation}).
We have the following result.

\begin{theorem}
\label{th:limit_SMLE}
Let the conditions of Theorem \ref{th:local_limit} be satisfied and let $h_n$ be a bandwidth such that $h_n\sim cn^{-1/5}$, as $n\to\infty$, for some $c>0$. Moreover, let the density $f_0$ be differentiable at $t\in(0,M_1)$, and let $K_h$ be defined by
\begin{align}
\label{K_h}
K_h(u)=h^{-1}K(u/h).
\end{align}
for the symmetric kernel $K$ which is the derivative of $\IK$. Finally, let the SMLE $\tilde{F}_{nh}$ be defined by (\ref{SMLE_incubation}). Then
\begin{align*}
n^{2/5}\left\{\tilde{F}_{n,h_n}(t)-F_0(t)\right\}\stackrel{d}\longrightarrow N(\mu,\s^2),
\end{align*}
where
\begin{align}
\label{mu_SMLE}
\mu=\tfrac12c^2f_0'(t)\int u^2K(u)\,du.
\end{align}
Moreover, if $Q_0$ is the probability measure of $(E_i,S_i)$,
\begin{align}
\label{sigma_SMLE}
\s^2=\lim_{n\to\infty}n^{-1/5}\|\th_{n,t,F_0}\|^2_{Q_0}=\lim_{n\to\infty}n^{-1/5}\int \th_{n,t,F_0}(e,s)^2\,dQ_0(e,s),
\end{align}
where
\begin{align*}
&\th_{n,t,F_0}(s,e)=\frac{\f_{n,t,F_0}(s)-\f_{n,t,F_0}(s-e)}{F_0(s)-F_0(s-e)},
\end{align*}
and the function $\f_{n,t,F_0}$ solves the integral equation
\begin{align}
\label{adjoint-SMLE}
&\int_{e>0}e^{-1}\left[\frac{\f(x+e)-\f(x)}{F_0(x+e)-F_0(x)}-\frac{\f(x)-\f(x-e)}{F_0(x)-F_0(x-e)}\right]\,dF_E(e)
=-K_{h_n}(t-x),\,\, x\in(0,M_1).
\end{align}
\end{theorem}

\vspace{0.5cm}
The proof is given in Section \ref{subsection:proof_Theorem5.1}.
We here give an outline of the proof. We consider:
\begin{align}
\label{adjoint_SMLE}
\int \IK_h(t-u)\,d\bigl(\hat F_n-F_0\bigr)(u)
&=-\int \th_{\hat F_n}(e,s)\,dQ_0(e,s),
\end{align}
where
\begin{align*}
&\th_{\hat F_n}(e,s)=\frac{\f_{\hat F_n}(s)-\f_{\hat F_n}(s-e)}{\hat F_n(s)-\hat F_n(s-e)},
\end{align*}
and where $\f_{\hat F_n}$ solves the integral equation 
\begin{align}
\label{adjoint-SMLE2}
\int_{e>0}e^{-1}\left[\frac{\f(v+e)-\f(v)}{\hat F_n(v+e)-\hat F_n(v)}-\frac{\f(v)-\f(v-e)}{\hat F_n(v)-\hat F_n(v-e)}\right]\,dF_E(e)=
-K_{h_n}(t-v), \quad v\in(0,M_1),
\end{align}
replacing $F_0$ by $\hat F_n$ in (\ref{adjoint-SMLE}).
Note that $\f_{\hat F_n}$ has both discrete and absolutely continuous parts. We do not have an explicit expression for $\f_{\hat F_n}$ or $\f_{n,t,F_0}$, but can compute it numerically, see \cite{github:20}.

Let $\Q_n$ the empirical measure of $(E_1,S_1),\dots,(E_n,S_n)$. 
The proof of the result can then be continued by proving
\begin{align*}
&-\int \th_{\hat F_n}(e,s)\,dQ_0\\
&=\int \th_{\hat F_n}(e,s)\,d\bigl(Q_n-Q_0\bigr)+o_p\left(n^{-2/5}\right)\\
&=\int \th_{n,t,F_0}(e,s)\,d\bigl(Q_n-Q_0\bigr)+o_p\left(n^{-2/5}\right),
\end{align*}
where we have a representation of $\int \IK_h(t-u)\,d\bigl(\hat F_n-F_0\bigr)(u)$ in terms of an integral in the observation space  $\int \th_{n,t,F_0}\,d(\Q_n-Q_0)$ in the last line, see Section \ref{subsection:proof_Theorem5.1}.

Using Theorem \ref{th:limit_SMLE}, we can construct confidence intervals, using the bootstrap, where we keep the exposure times $E_i$ fixed. Our bootstrap sample consists of:
\begin{align*}
(E_1,S_1^*),\dots,(E_n,S_n^*),
\end{align*}
where
\begin{align}
\label{bootstrap_S_i^*}
S_i^*=U_i^*+V_i^*,
\end{align}
and where $U_i^*$ is uniform on $[0,E_i]$ and $V_i^*$ is generated from the SMLE of the incubation time with a bandwidth $h_0$ of order $n^{-1/9}$. The oversmoothing is used to deal with the bias (see below) and was introduced in \cite{marron:91} in the context of nonparametric regression analysis.

The random variables $V_i^*$ are computed by generating Uniform$(0,1)$ random variables $W_i$ and solving the equation
\begin{align}
\label{golden-section}
\tilde F_{nh_0}(x)=W_i,
\end{align}
in $x$, using golden-section search, and taking $V_i^*=x$ for such $x$, solving (\ref{golden-section}).  See the code {\tt bootstrap{\_}SMLE.cpp} in \cite{github:20}, which is used in the corresponding {\tt R} script {\tt bootstrap{\_}SMLE.R}.

The $95\%$ bootstrap confidence intervals are given by
\begin{align}
	\label{first_conf_intervals}
	\left(\tilde F_{nh}(t)-Q_{0.975}^*,\tilde F_{nh}(t)-Q_{0.025}^*\right),
\end{align}
where $Q_{0.025}^*$ and $Q_{0.975}^*$ are the $2.5$th and $97.5$th percentiles of the values of
\begin{align*}
\tilde F_{nh}^*(t)-\tilde F_{nh_0}(t)
\end{align*}
for $1000$ (bootstrap) samples of (\ref{bootstrap_S_i^*}), and where $\tilde F_{nh}^*$ is the SMLE with bandwidth $h$, corresponding to the MLE $\hat F_n^*$ computed for a bootstrap sample of size $n$.

\begin{figure}[!ht]
\includegraphics[width=0.5\textwidth]{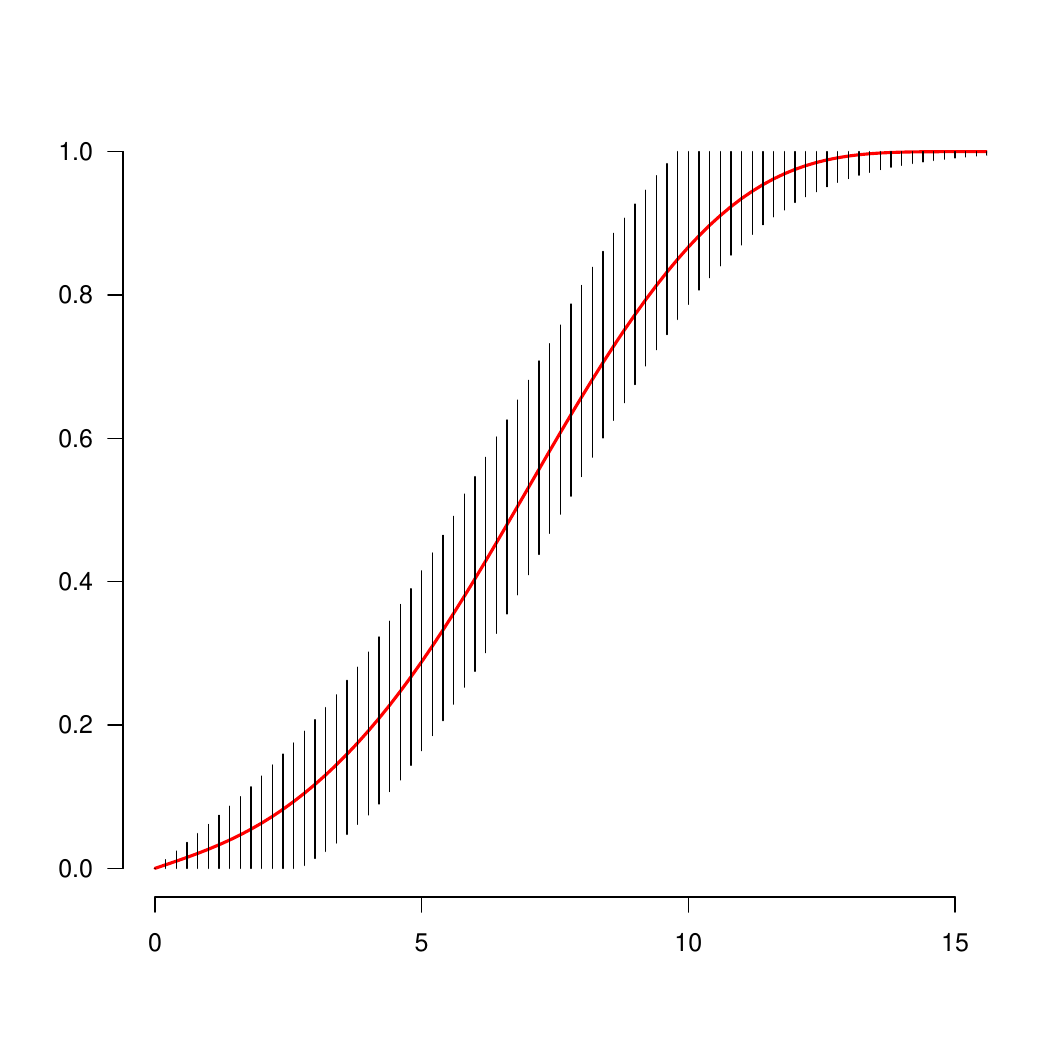}
\caption{Smoothed bootstrap onfidence intervals for the Wuhan data set, based on the SMLE. The red curve is the SMLE $\tilde F_{n,h}$, where $h=6.21628$.}
\label{figure:CI_incubation_df}
\end{figure}

\begin{example}
{\rm
We consider again the data on the Wuhan travelers, considered in Example \ref{example_Wuhan}.
For Figure \ref{figure:CI_incubation_df} all $1000$ values $\tilde F_{nh}^*(t)-\tilde F_{n, h_0}(t)$, and the percentiles $Q^*_{0.025}(t)$ and $Q^*_{0.975}(t)$ were determined. This gives the bootstrap intervals (\ref{first_conf_intervals}).
Figure \ref{figure:CI_incubation_df} can be reproduced by running the {\tt R} script {\tt bootstrap{\_}SMLE.R} in \cite{github:20}.
}
\end{example}

The following result shows that we can expect the SMLE's computed on the basis of the bootstrap samples to behave asymptotically in the same  way as the original SMLE's.

\begin{theorem}
\label{th:bootstrap_SLSE}
Let the conditions of Theorem \ref{th:limit_SMLE} be satisfied. Moreover, let $h\sim cn^{-1/5}$ and $h_0\sim c_0n^{-1/9}$, for some positive constants $c$ and $c_0$. Then, at $t\in(0,M_1)$,
\begin{align*}
n^{2/5}\left\{\tilde F_{nh}^*(t)-\tilde F_{nh_0}(t)\right\} \stackrel{\cal D}\longrightarrow N(\mu,\s^2), 
\end{align*}
given $(E_1,S_1),\dots,(E_n,S_n)$, almost surely along sequences $(E_1,S_1),(E_2,S_2),\dots$, where $\mu$ and $\s^2$ are defined by (\ref{mu_SMLE}) and (\ref{sigma_SMLE}).
\end{theorem}

The proof of this result follows the lines of the proof of Theorem \ref{th:limit_SMLE} in Section \ref{subsection:proof_Theorem5.1}, apart from the fact that in this case the distribution function generating the (bootstrap) incubation times is $\tilde F_{nh_0}$ (instead of $F_0$), which converges almost surely to $F_0$ and that therefore Lindeberg-Feller conditions are used for the central limit theorem because of the varying underlying measure. This matter is discussed in the context of a nonparametric regression model in \cite{piet_geurt:23}.

Note that the centering constant $\m$ arises from the difference
\begin{align*}
\int K_h(t-u)\,\tilde F_{nh_0}(u)\,du-\tilde F_{nh_0}(t)\sim\tfrac12h^2\tilde F_{nh_0}''(t)\sim\tfrac12h^2F_0''(t), 
\end{align*}
and that we need here that $h_0$ tends to zero slower than $n^{-1/5}$, see section \ref{subsection:bandwidth_choice} below.

\subsection{Bandwidth selection}
\label{subsection:bandwidth_choice}
We propose a bootstrap method to find an approximately MSE optimal bandwidth for estimating $F_0(t)$ at a point $t\in(0,M_1)$. The MSE we want to minimize as a function of $h$ is given by:
\begin{align}
\label{MSE1}
MSE_h(t)=\E \bigl\{\tilde{F}_{nh}(t)-F_0(t)\bigr\}^2.
\end{align}
The analogous bootstrap quantity (using oversmoothing, in the sense that $h_0\asymp n^{-1/9}$) is given by:
\begin{align}
\label{MSE_bootstrap}
MSE_h^*(t)=\E\left\{\left\{\tilde{F}_{nh}^*(t)-\tilde F_{nh_0}(t)\right\}^2\Bigm|(E_1,S_1),\dots,(E_n,S_n)\right\},
\end{align}
where $h_0$ is called a ``pilot'' bandwidth. We shall show that (\ref{MSE_bootstrap}) is asymptotically independent of the constant $c_0$ in the pilot bandwidth $h_0$ if we take $h_0=c_0n^{-1/9}$.

Using integration by parts, we can write
\begin{align*}
\tilde{F}_{nh}^*(t)-\tilde F_{nh_0}(t)&=\int K_h(t-x)\,\bigl\{\hat F_n^*(x)-\tilde F_{nh_0}(x)\bigr\}\,dx\\
&\qquad+\int K_h(t-x)\,\tilde F_{nh_0}(x)\,dx-\tilde F_{nh_0}(t).
\end{align*}

We have (MSE is variance $+$ squared bias):
\begin{align}
\label{variance_bias_decomp}
	MSE_h^*(t)&\sim\E\left\{\left\{\int K_h(t-x)\,\left\{\hat F_n^*(x)-\tilde F_{nh_0}(x)\right\}\,dx\right\}^2\Bigm|(E_1,S_1),\dots,(E_n,S_n)\right\}\nonumber\\
&\qquad\qquad	+\left\{\int K_h(t-x)\, \tilde F_{nh_0}(x)\,dx- \tilde F_{nh_0}(t)\right\}^2.
\end{align}
For the second term on the right we get:
\begin{align*}
\int K_h(t-x)\,\tilde F_{nh_0}(x)\,dx-\tilde F_{nh_0}(t)
=\tfrac12h^2\tilde F_{nh_0}''(t)\int u^2K(u)\,du+o_p\left(h^2\right),
\end{align*}
so
\begin{align*}
&\E\left\{\left\{\int K_h(t-x)\, \tilde F_{nh_0}(x)\,dx- \tilde F_{nh_0}(t)\right\}^2\Bigm|(E_1,S_1),\dots,(E_n,S_n)\right\}\\
&=\tfrac14h^4\tilde F_{nh_0}''(t)^2\left\{\int u^2K(u)\,du\right\}^2+o_p\left(h^4\right).
\end{align*}

We have the following result.

\begin{lemma}
\label{lemma:bias_term}
Let the conditions of Theorem \ref{th:limit_SMLE} be satisfied. Moreover, let $h_0=h_{n,0}\sim c_0n^{-1/9}$, as $n\to\infty$. Then
\begin{align*}
&\tilde F_{nh_0}''(t)\stackrel{p}\longrightarrow f_0'(t),\qquad n\to\infty.
\end{align*}
\end{lemma}

\begin{figure}[!ht]
\includegraphics[width=0.5\textwidth]{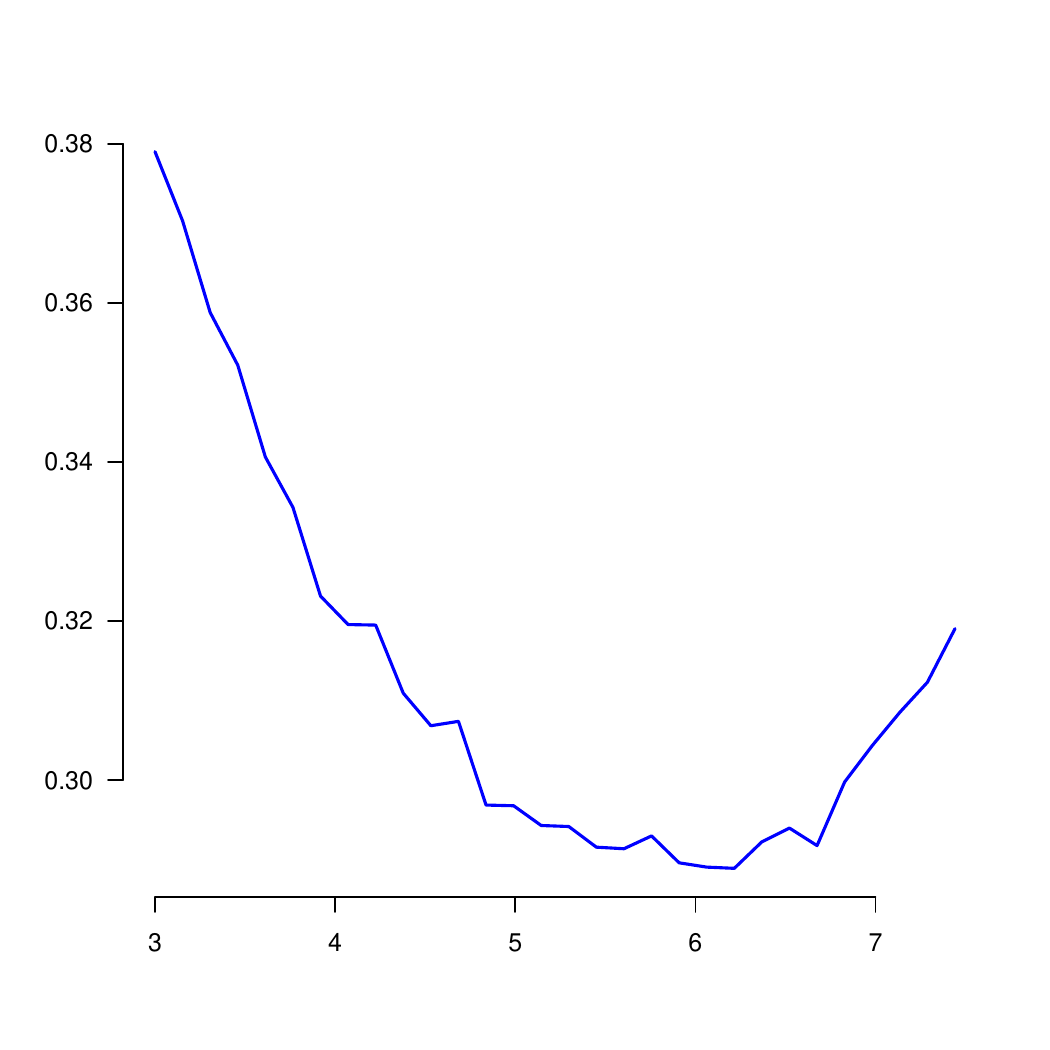}
\caption{The bootstrap MSE for the Wuhan data.}
\label{figure:MSE}
\end{figure}

For reasons of space we omit the proof here. More details on this type of lemma can be found in \cite{piet_geurt:23} in the context of monotone regression.

\begin{remark}
{\rm
Note that this convergence result does not hold if the pilot bandwidth $h_0$ is of order $n^{-1/5}$. For this reason the method suggested in \cite{SenXu2015}, where the pilot bandwidth is chosen of order $n^{-1/5}$  will not work, since in that case the variance of $\tilde F_{nh_0}''(t)$ will not tend to zero. Another way out is to use subsampling, as used in
\cite{kim_piet:18:SJS}, but choosing the right subsample size is a rather hard problem.

Since the first order behavior of the first term of (\ref{variance_bias_decomp}) also only depends on $h$ and $n$ and not on $h_0$, the dependence on the constant $c_0$ in the pilot bandwidth disappears in first order, as $n\to\infty$, which indicates robustness of this bandwidth choice procedure. In fact, $n^{4/5}$ times the first term of (\ref{variance_bias_decomp}) tends to the limit variance (\ref{sigma_SMLE}) and $n^{4/5}$ times the second term of (\ref{variance_bias_decomp}) tends to the squared bias $\mu^2$, where $\m$ is defined by (\ref{mu_SMLE}), irrespective of the constant $c_0$ in the pilot bandwidth $h_0=c_0n^{-1/9}$.
}
\end{remark}

Instead of minimizing (\ref{MSE_bootstrap})  we minimize a  Monte Carlo approximation of a modified version of (\ref{MSE_bootstrap}):
\begin{align}
\label{bootstrap_L2-distance}
B^{-1}\sum_{i=1}^B \sum_{j=1}^m\left\{\tilde F^{*,i}_{nh}(t_j)-\tilde F_{nh_0}(t_j)\right\}^2,
\end{align}
where the $\tilde F^{*,i}_{nh}(t_j)$, $i=1,\dots,B$, $t_j,\dots,m$ are the estimates in $B$ bootstrap samples on a grid of equidstant points $t_j$. The script for this procedure is given by the {\tt R} script {\tt bandwidth{\_}choice{\_}df.R} on \cite{github:20}. This script produced the bandwidth in Figure \ref{figure:CI_incubation_df}. We took $B=10,000$ in this case.

A picture of the bootstrap MSE (\ref{bootstrap_L2-distance}) as a function of $h$ for the data on the travelers from Wuhan is shown in Figure \ref{figure:MSE}.

\section{Estimation of quantiles and comparison with parametric methods}
\label{section:comparison}
We now illustrate the difference between the nonparametric approach and the approach using distributions like the Weibull, log-normal, etc.\ for the incubation time distribution.  This is shown for the problem of estimating the 95th percentile of the distribution. To this end we generated $1000$ samples of size $n=500$ and also size $n=1000$, using the same Weibull distribution to generate the incubation time distribution as we used in Section \ref{section:bootstrap_incubation_time} for constructing confidence intervals. This example is also given (for sample size $n=500$) in \cite{pietNAW:21}.

In the (truncated) Weibull approach to the problem, we maximize for $\a,\b>0$:
\begin{align}
\label{log_l_weibull}
\sum_{i=1}^n \log\left\{F_{\a,\b}(S_i)-F_{\a,\b}(S_i-E_i)\right\},
\end{align}
where $F_{\a,\b}$ is defined by (\ref{weibull_df}).
 This gives a maximum likelihood estimate $F_{\hat\a,\hat\b}$ of the distribution function, where $(\hat\a,\hat\b)$ maximizes (\ref{log_l_weibull}) over $(\a,\b)$.  The estimate of the 95th percentile is then defined by $F^{-1}_{\hat\a,\hat\b}(0.95)$, where $F^{-1}_{\hat\a,\hat\b}$ denotes the inverse function.

In the log-normal approach to the problem, we maximize for $\a\in\R$ and $\b>0$:
\begin{align}
\label{log_l_lognormal}
\sum_{i=1}^n \log\left\{G_{\a,\b}(S_i)-G_{\a,\b}(S_i-E_i)\right\},
\end{align}
where $G_{\a,\b}$ is defined by 
\begin{align}
\label{log-normal_df}
G_{\a,\b}(x)=\Phi\left((\log x - \a)/\b\right),
\end{align}
for $x> 0$ (zero otherwise), where $\b>0$ and $\Phi$ is the standard normal distribution function.
The estimate of the percentile is then given by $G^{-1}_{\hat\a,\hat\b}(0.95)$, where $(\hat\a,\hat\b)$ maximizes (\ref{log_l_lognormal}) over $(\a,\b)$.

In the {\it nonparametric} maximum likelihood approach we simply maximize 
\begin{align*}
\sum_{i=1}^n \log\left\{F(S_i)-F(S_i-E_i)\right\},
\end{align*}
over {\it all} distribution functions $F$. This give the nonparametric MLE $\hat F_n$, from which we compute the SMLE $\tilde F_{n,h_n}(t)=\int \IK_h(t-y)\,d\hat F_n(y)$ and the estimate of the 95th percentile $\tilde F^{-1}_{n,h_n}(0.95)$. The bandwidth $h$ was chosen to be $h_n=6n^{-1/5}$ here. Note that, using the delta method, we find:
\begin{align*}
&n^{2/5}\left\{\tilde F_{n,h_n}^{-1}(0.95)-F_{\a,\b}^{-1}(0.95)\right\}\\
&=-n^{2/5}\left\{\tilde F_{n,h_n}(F_{\a,\b}^{-1}(0.95))-0.95)\right\}\bigm/f_{\a,\b}\left(F_0^{-1}(0.95)\right)+o_p(1),
\end{align*}
where $f_{\a,\b}$ is the (truncated) Weibull density, corresponding to the distribution function $F_{\a,\b}$, defined by (\ref{weibull_df}).

The results of this simulation for $1000$ samples of size $n=500$ are shown in the box plot Figure \ref{figure:boxplot} and the corresponding picture for sample size $n=1000$ in Figure \ref{figure:boxplot2}.

\begin{figure}[!ht]
\begin{center}
\includegraphics[width=0.55\textwidth]{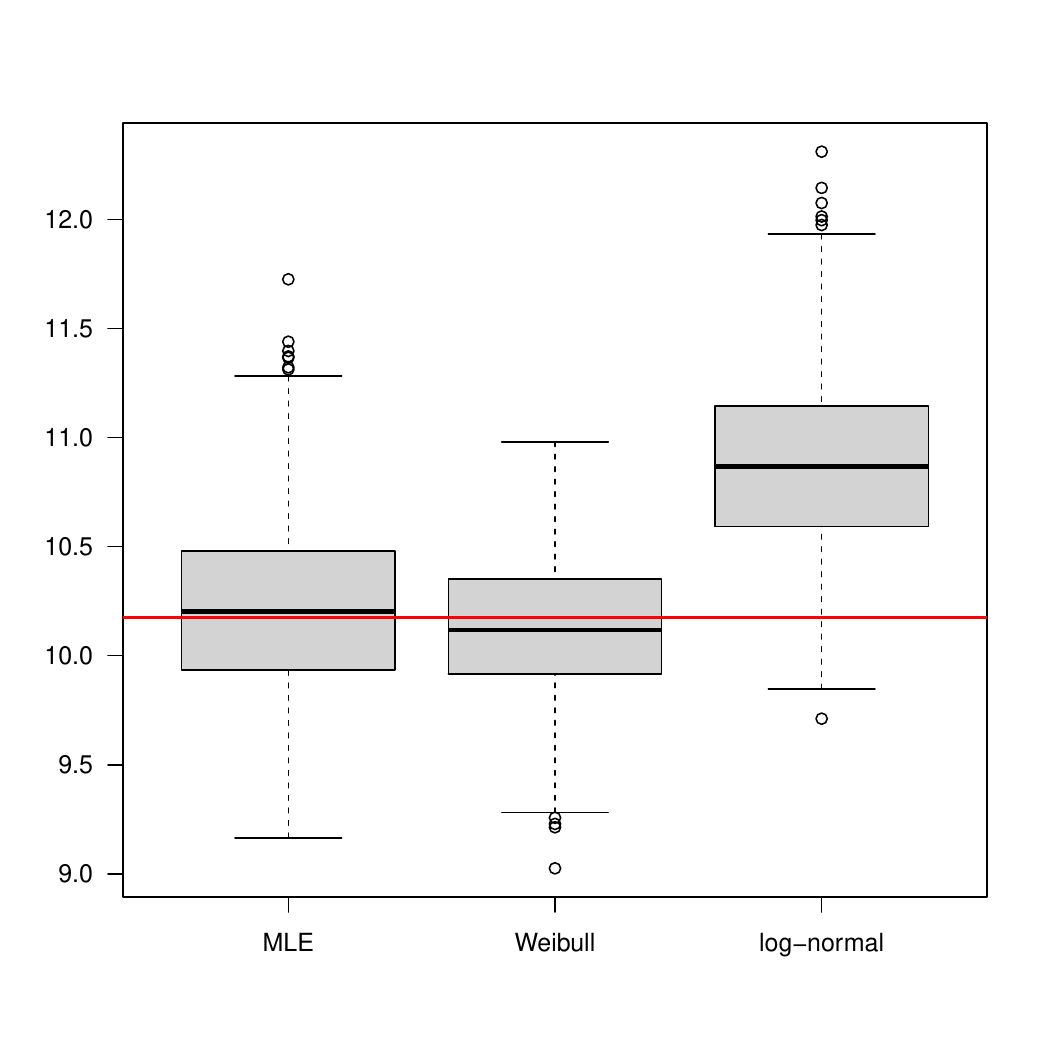}
\captionof{figure}{Box plot of 95th percentile estimates for the nonparametric, Weibull and log-normal maximum likelihood estimators for $1000$ samples of size $n=500$. The incubation time data are generated from a Weibull distribution. The red line denotes the value of the true percentile.}
\label{figure:boxplot}
\end{center}
\end{figure}

\begin{figure}[!ht]
\begin{center}
\includegraphics[width=0.55\textwidth]{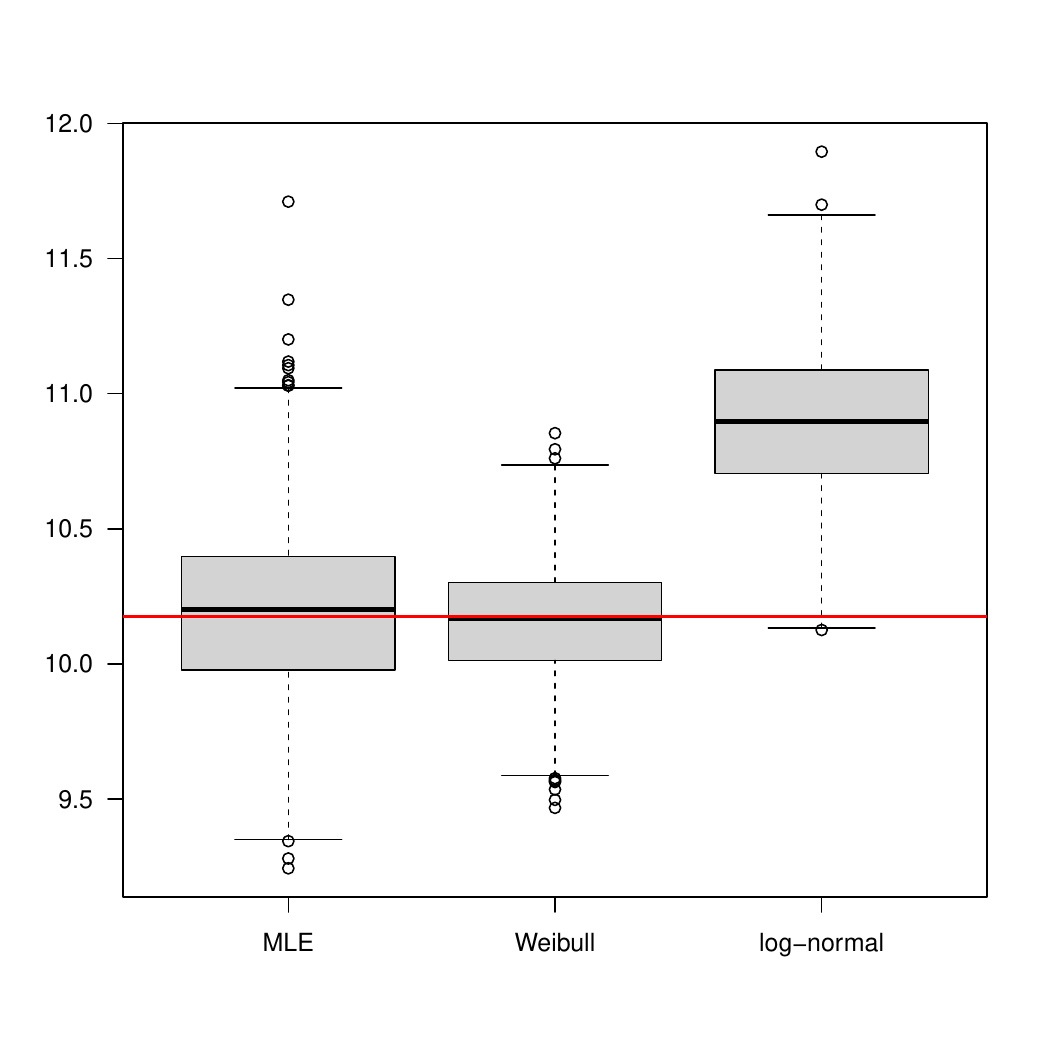}
\captionof{figure}{Box plot of 95th percentile estimates for the nonparametric, Weibull and log-normal maximum likelihood estimators for $1000$ samples of size $n=1000$. The incubation time data are generated from a Weibull distribution. The red line denotes the value of the true percentile.}
\label{figure:boxplot2}
\end{center}
\end{figure}

The black line segments in the boxes are at the position of the median. Finally, the red line denotes the value of $F_{\a,\b}^{-1}(0.95)\approx10.17716$, where $(\a,\b)$ are the parameters of the Weibull distribution. {\tt R} scripts for all methods are given in the directory ``simulations'' of \cite{github:20}.

It can be seen that, since the incubation time data were generated from a Weibull distribution, the estimates of the quantiles assuming this distribution have indeed the smallest variation. But the nonparametric estimates, not making the assumption that the distribution is of the Weibull type, are also pretty good, whereas the  estimates, assuming a log-normal distribution are completely off (in fact, these estimate are inconsistent). The SMLE adapts to the underlying distribution and provides consistent estimates, using the consistency of the MLE itself, derived in Section \ref{section:consistency} and the consistency of the SMLE, which can be deduced from this.

One sees that the uncertainty about what interest us, is not much larger when one only uses the nonparametric SMLE than when one assumes that the incubation time distribution is Weibull (which is the correct distribution in this simulation setting), but much larger when one assumes log-normal. While there is absolutely no scientific (medical) reason to “believe” Weibull, or to “believe” log-normal. They lead to completely different statistical inferences, hence could lead to completely different policy recommendations.

\section{Other smooth functionals}
\label{section:smooth_functionals}
The first moment is the prototype of a smooth functional, The asymptotic normality and $\sqrt{n}$ convergence  of the estimate
\begin{align*}
\int x\,d\hat F_n(x)
\end{align*}
where $\hat F_n$ is the nonparametric MLE was given for the current status model in \cite{GrWe:92}, Theorem 5.5 of Part 2. The asymptotic variance is given by
\begin{align*}
\s^2=\int\frac{F_0(t)\{1-F_0(t)\}}{g(t)}\,dt,
\end{align*}
where $g$ is the density of the observation times and $F_0$ the distribution function of the hidden estimate.

Similar results for more general cases of interval censoring are given in Chapter 10 of \cite{piet_geurt:14}, but in those cases the expression for the asymptotic varaiance is coming from the solution of an integral equation and no longer explicit as in the case of the current status model. A similar situation holds for the model for the incubation time distribution.

We have the following asymptotic normality result for the estimate of the first moment, based on the nonparametric MLE $\hat F_n$ for the incubation time model, if the support of the incubation time distribution is $[0,M_1]$:
\begin{align}
\label{convergence_mean}
\sqrt{n}\left\{\int x\,d\hat F_n(x)-\int x\,dF_0(x)\right\}\stackrel{{\cal D}}\longrightarrow N(0,\s^2),
\end{align}
where $N(0,\s^2)$ is a normal distribution with mean zero and variance
\begin{align*}
\s^2=-\int_0^{M_1}\f_{F_0}(x)\,dx.
\end{align*}
and $\f_{F_0}$ is also the solution of the following equation in $\f$:
\begin{align}
\label{phi-diff}
\int_{e>0}e^{-1}\left[\frac{\f(v+e)-\f(v)}{F(v+e)-F(v)}-\frac{\f(v)-\f(v-e)}{F(v)-F(v-e)}\right]\,dF_E(e)=1,\qquad v\in[0,M_1].
\end{align}
The distribution function $F_E$ of the exposure time was again chosen to be the uniform distribution function on $[1,30]$.

\begin{figure}[!ht]
\begin{center}
\includegraphics[width=0.6\textwidth]{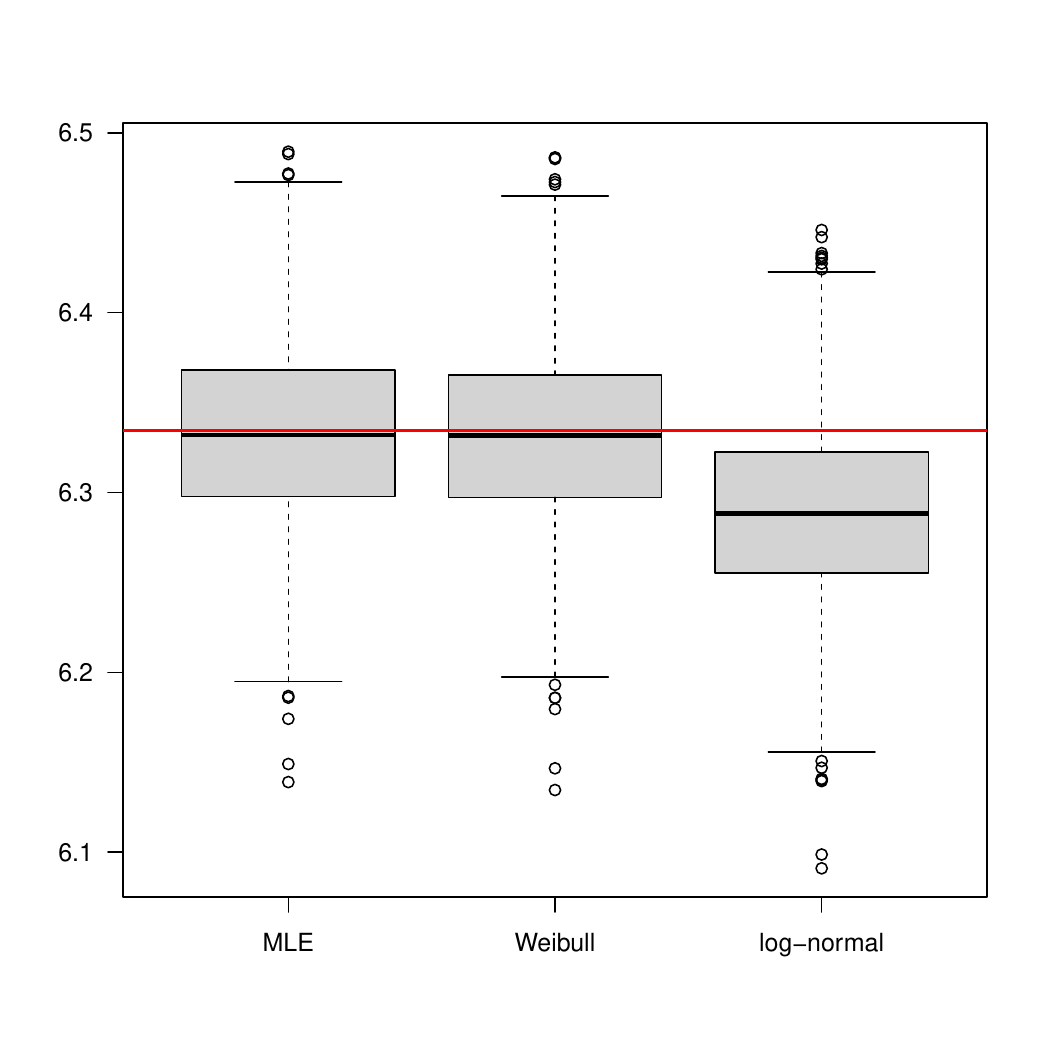}
\captionof{figure}{Box plot of estimation of the first moment of the incubation distribution for the nonparametric, Weibull and log-normal maximum likelihood estimators for $1000$ samples of size $n=5000$. The incubation time data are generated from a Weibull distribution. The red line denotes the value of the actual real first moment.}
\label{figure:boxplot_mean}
\end{center}
\end{figure}

The derivation of this result is given in the Appendix. The inconsistency of the estimate based on the log normal model is again clearly seen from the boxplot Figure \ref{figure:boxplot_mean}. However, if we would have generated the incubation time distribution from a log normal distribution, the estimate based on the Weibull distribution would be inconsistent, so Figure \ref{figure:boxplot_mean} cannot be interpreted as showing the superiority of the Weibull distribution.

Examples of the behavior of the density estimate
\begin{align*}
\hat f_{nh}(t)=\int K_h(t-y)\,d\hat F_n(y),
\end{align*}
which converges at rate $n^{2/7}$, where $h\sim c n^{-1/7}$, $c>0$, are given in \cite{piet:21} and \cite{pietNAW:21}

\section{Conclusion}
\label{section:conclusion} 
We proved that the nonparametric MLE in a model for the incubation time distribution converges in distribution, after standardization, to Chernoff's distribution. The rate of convergence is cube root $n$, if $n$ is the sample size, under a separation condition for the exposure time.  We also discussed (locally) differentiable functionals of the model, estimated by corresponding functionals of the nonparametric MLE, which converge after standardization to a normal distribution at faster rates, where the constants are given by the solution of an integral equation.

This provides an alternative for the parametric models that are usually applied in this context, estimating the incubaton time distribution by, e.g., Weibull, gamma or log-normal distributions.
If the parametric model is not right (there is in fact no scientific or medical reason to choose for Weibull, gamma, Erlang, log-normal. etc., and one sees for this reason usually these distribution applied at the same time) the estimates are inconsistent if the chosen model does not hold, as we demonstrate in Sections \ref{section:comparison} and \ref{section:smooth_functionals}

As shown in Section \ref{section:smooth_functionals}, for parameters like the first moment, we do not have to choose a bandwidth parameter, while the behavior of the estimate based on the nonparametric MLE is competitive to the parametric estimates in this case, even if the model for the parametric estimate is right.

{\tt R} scripts for computing the estimates are given in \cite{github:20}.

\section{Appendix}
\label{section:Appendix}

\subsection{Integral equations}
\label{subsection:Integral_equations}
As explained in the Appendix of \cite{piet:21}, the theory of the estimation of smooth functionals in the model is based on certain integral equations. For the incubation time model an extra indicator $\dd_i$ was introduced to indicate whether $S_i\le E_i$.

But introducing such an indicator is not necessary. So we define the score function $\th$ without these idicators and get as definition for the score function:
\begin{align*}
\th(e,s)=E\left\{a(X)|(E,S)=(e,s)\right\}=\frac{\int a(x)\,dF(x)}{F(s)-F(s-e)}\,, 
\end{align*}
(compare to (A.1) in \cite{piet:21}). This is the conditional expectation of $a(X)$ in the ``hidden'' space of the variable of interest (the incubation time), given our observation $(E,S)$. Note that we changed the notation somewhat w.r.t.\ \cite{piet:21}, and denote the distribution function of the incubation time by $F$ instead of $G$.

Defining
\begin{align*}
\f(t)=\int_{x\in[0,t]}a(x)\,dF(x),
\end{align*}
we get the following representation for the score function, conditioned on $X=x$:
\begin{align}
\label{original_irepresentation}
E\{\th(E,S)|X=x\}
=\int_{e\in[0,M_2]}e^{-1}\left\{\int_{s\in(x,x+e]}\frac{\f(s)-\f(s-e)}{F(s)-F(s-e)}\,ds\right\}\,dF_E(e).
\end{align}
Note that in \cite{piet:21} the distribution function $F_E$ was taken to be uniform, just as an example, but that we do not assume that here.

\vspace{0.4cm}
Differentating (\ref{original_irepresentation}) w.r.t.\ $x$, we get the following integral equation, with on the right the derivative of the functional
 we want to estimate, denoted by $\psi$:
\begin{align}
\label{phi-int}
&\int e^{-1}\left[\frac{\f(x+e)-\f(x)}{F(x+e)-F(x)}-\frac{\f(x)-\f(x-e)}{F(x)-F(x-e)}\right]\,dF_E(e)\nonumber\\
&=-\frac{\f(x)}{1-F(x)}\int_{e\ge M_1-x}e^{-1}\,dF_E(e)-\frac{\f(x)}{F(x)}\int_{e\ge x}e^{-1}\,dF_E(e)\nonumber\\
&\qquad+\int e^{-1}\left[\frac{\f(x+e)-\f(x)}{F(x+e)-F(x)}1_{\{e<M_1-x\}}-\frac{\f(x)-\f(x-e)}{F(x)-F(x-e)}1_{\{e<x\}}\right]\,dF_E(e)\nonumber\\
&=\psi(x),\qquad x\in[0,M_1].
\end{align}
This is called a Fredholm integral equation of the second kind. Here and in the sequel, we assume that distribution functions are right-continuous.

We assume that $F_E$ satisfies the following condition:
\begin{enumerate}
\item[(F1)] $F_E$ is a distribution function on $[0,M_2]$, $M_2>M_1/2$, which is $0$ on $[0,\e]$ for some $\e\in (0,(M_1\wedge M_2)/2)$ and has a continuous strictly positive derivative $f_E$ on $[\e,M_2]$. The distribution function is defined to be $1$ on $[M_2,\infty)$.
\end{enumerate}
We define the class of distribution functions ${\cal F}$ in the following way.
\begin{enumerate}
\item[(F2)] ${\cal F}$ consists of the distribution functions $F$ on $[0,M_1]$ with only a finite number of jumps, contained  in $(0,M_1)$, and satisfying
\begin{align}
\label{separation_condition}
F(u)-F(t)\ge c>0, \qquad\text{ if } u-t\ge \e\quad\text{and}\quad t,u\in[0,M_1].
\end{align}
where $\e>0$ is defined as in Definition (F1). The distribution functions are extended to $[0,\infty)$ by defining it to be equal to $1$ on $[M_1,\infty)$.
\end{enumerate}

Note that the MLE satisfies the conditions of (F2) for sufficiently large $n$, by the conditions of Theorem \ref{th:local_limit} (in particular the fact that the density $f_0$ is strictly positive on $(0,M_1)$) and the consistency of the MLE (Theorem \ref{th:consistency}).

\begin{lemma}
\label{lemma:uniqueness_sol_inteq}
Let $F_E$ be a distribution function on $[0,M_1]$, satisfying condition (F1) and let $F\in{\cal F}$, where ${\cal F}$ is defined in (F2).               
Moreover, let, $F\in{\cal F}$ satisfy:
\begin{align}
\label{side_condition}
\{1-F(t)\}\int_{e\ge t}e^{-1}\,dF_E(e)+F(t)\int_{e\ge M_1-t}e^{-1}\,dF_E(e)\ge c>0
\end{align}
for some $c>0$ and each $t\in[0,M_1]$.

Finally, let $\psi:[0,M_1]\to\R$ be a bounded right-continuous function with left-limits (cadlag) on $[0,M_1]$. Then the equation (\ref{phi-int}) has a unique solution in $\f:[0,M_1]\to\R$.
\end{lemma}

\begin{proof}
Consider, for $t\in[0,M_1]$, the homogeneous equation
\begin{align}
\label{homogeneous_eq}
&\f(t)\left\{\{1-F(t)\}\int_{e\ge t}e^{-1}\,dF_E(e)+F(t)\int_{e\ge M_1-t}e^{-1}\,dF_E(e)\right\}\nonumber\\
&=F(t)\{1-F(t)\}\nonumber\\
&\qquad\qquad\cdot\int e^{-1}\left[\frac{\f(t+e)-\f(t)}{F(t+e)-F(t)}1_{\{e<M_1-t\}}-\frac{\f(t)-\f(t-e)}{F(t)-F(t-e)}1_{\{e<t\}}\right]\,dF_E(e).
\end{align}
Since $F(0)=0$ and $F(M_1)=1$, the solution has to be zero at $t=0$ and $t=M_1$.
Suppose there exists a point $t\in(0,M_1)$ such that $\f(t)>0$ for a solution $\f$. If the maximum of $\f$ is reached at the point $s\in(0,M_1)$, we get that the right-hand side of (\ref{homogeneous_eq}) is nonpositive at $t=s$, whereas the left-hand side is $>0$, a contradiction. Note that we use condition (\ref{side_condition}).

If the supremum is not attained, we take the limit from the left and get that the limit on the left is strictly positive, whereas the limit on the right in nonpositive, which leads again to a contradiction.

A similar argument holds if there exists a point $t$ such that $\f(t)<0$. It now follows from Theorem 3.4 in \cite{kress:89} that the integral equation has a unique solution (see also Lemma 10.1 on p.\ 294 of \cite{piet_geurt:14}).
\end{proof}

\begin{figure}[!ht]
\centering
\includegraphics[width=0.5\textwidth]{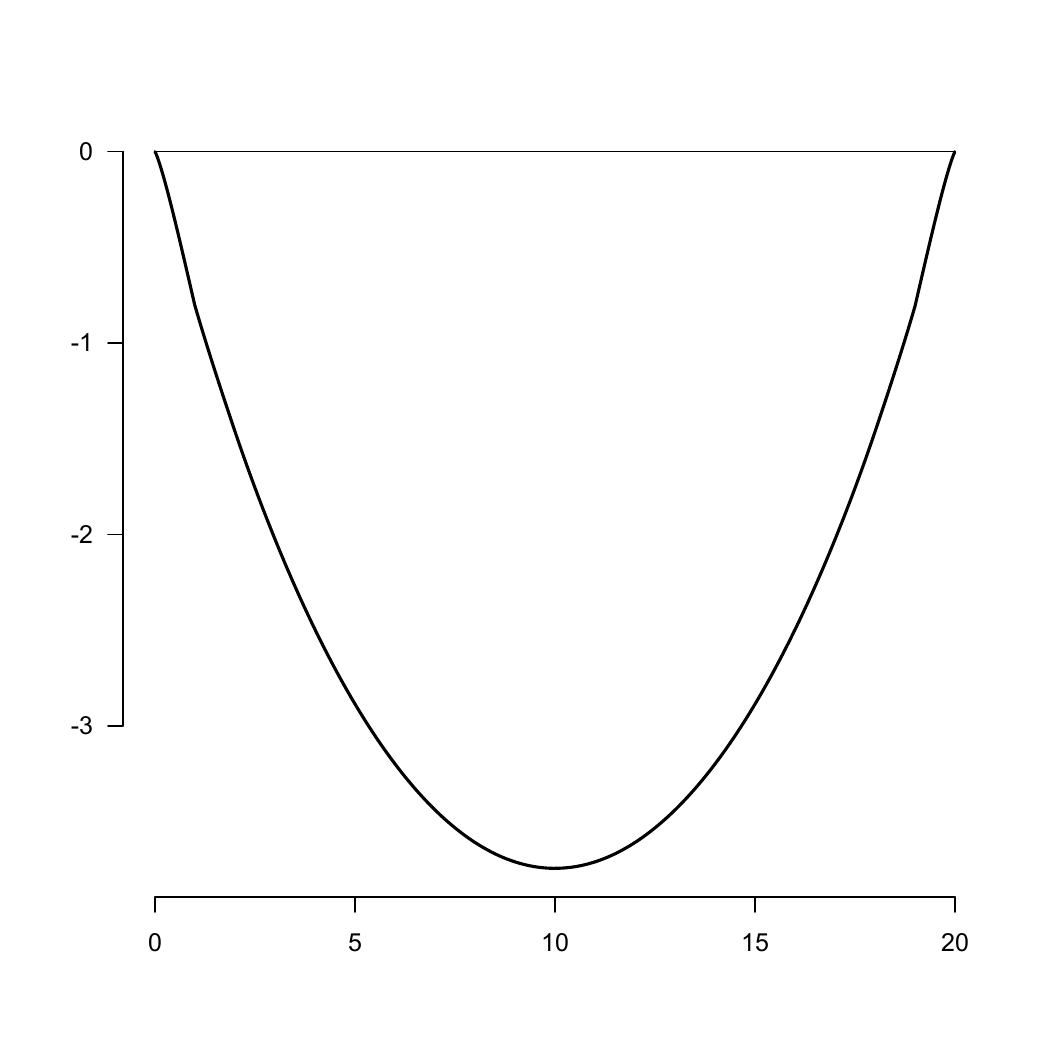}
\caption{The solution of (\ref{phi-int}) if $F_0$ is uniform on $[0,20]$ and $F_E$ is uniform on $[1,30]$.}
\label{figure:phi_mean}
\end{figure}

\begin{example}
{\rm
We consider the example $\psi\equiv1$, $F_E$ is the uniform distribution on $[1,30]$ and $F_0$ is uniform on $[0,20]$. In that case all conditions of Lemma \ref{lemma:uniqueness_sol_inteq} are satisfied. One could think of an exposure time of at most $30$ days and at least one day in the incubation time model.
It is shown in Section \ref{subsection:proof_8.1} that
\begin{align*}
\sqrt{n}\left\{\int x\,d\hat F_n(x)-\int x\,dF_0(x)\right\}\stackrel{{\cal D}}\longrightarrow N(0,\s^2),
\end{align*}
where
\begin{align*}
\s^2=-\int_0^{M_1}\f_{F_0}(x)\,dx,
\end{align*}
where $\f_{F_0}$ solves (\ref{phi-int}) for $\psi\equiv1$ and $F=F_0$, where we choose in the present example $F_0$ to be the uniform distribution on $[0,M_1]$.
We can solve equation (\ref{phi-int}) approximately by a matrix equation, see \cite{github:20}.
}
\end{example}

\begin{example}
{\rm
If $\psi=1_{[0,a)}$, where $a\in(0,M_1)$, $\psi$ is not continuous on $[0,M_1]$, but still satisfies the conditions of the lemma. In this case the solution has a jump at $a$, as is shown in Figure \ref{figure:phi_discont} for $a=10$.

\begin{figure}[!ht]
\centering
\includegraphics[width=0.5\textwidth]{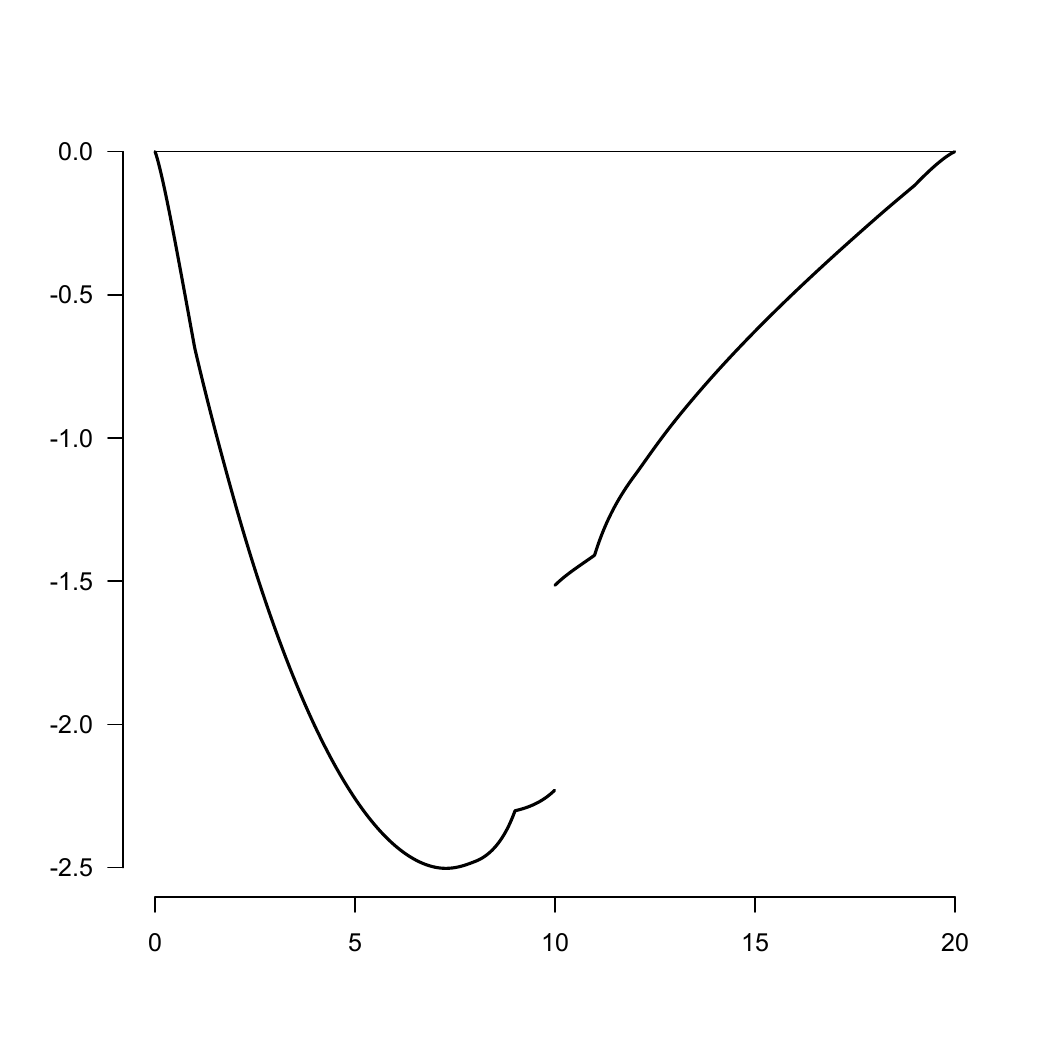}
\caption{The solution for $\psi=1_{[0,10)}$ if $F_0$ is uniform on $[0,20]$ and $F_E$ is uniform on $[1,30]$.}
\label{figure:phi_discont}
\end{figure}

This corresponds to the functional
\begin{align*}
x\mapsto \Psi_{a,F}(x)=g_a(x)-\int g_a(x)\,dF(x),
\end{align*}
where the function $g_a$ is given by
\begin{align*}
g_a(x)=x1_{[0,a)}(x)+a\,1_{[a,M_1]}(x),\qquad x\in[0,M_1].
\end{align*}
Note that $g_a$ is continuous, a fact that is absolutely essential in the smooth functional theory.

Integration by parts gives:
\begin{align*}
&\int g_a(x)\,dF_0(x)=aF_0(a-)-\int_0^a F_0(x)\,dx+a\{1-F_0(a)\}=
\int_0^a \{1-F_0(x)\}\,dx.
\end{align*}
Hence we get generally:
\begin{align*}
\int\Psi_{a,F}(x)\,d\bigl(F-F_0\bigr)(x)=-\int\Psi_{a,F}(x)\,dF_0(x)=-\int_0^a\{F(x)-F_0(x)\}\,dx
\end{align*}
So in studying expressions like $\int_0^{a}\{\hat F_n(x)-F_0(x)\}\,dx$, for $a\in(0,M_1]$, where $\hat F_n$ is the nonparametric MLE and $F_0$ the underlying distribution function, and in showing things like
\begin{align*}
\sqrt{n}\sup_{a\in[0,M_1]}\left|\int_0^{a}\{\hat F_n(x)-F_0(x)\}\,dx\right|=O_p(1),
\end{align*}
we need a functional of this type in the proof of the limit result for the MLE in the incubation time model.
}
\end{example}

\subsection{Properties of the solution of the integral equation (\ref{phi-int})}
\label{subsection:properties_integral equations}
The integral equation (\ref{phi-int}) has the same structure as the integral equations for the interval censoring, case 2, model. This equation is given by (10.33) on p.\ 292 of \cite{piet_geurt:14}.

Let, for $F\in{\cal F}$, the function $d_F$ be defined by
\begin{align}
\label{d_F}
d_F(t)=
\{1-F(t)\}\int_{e\ge t}e^{-1}\,dF_E(e)+F(t)\int_{e\ge M_1-t}e^{-1}\,dF_E(e).\qquad t\in[0,M_1].
\end{align}
We have:
\begin{align}
\label{lower_bound_d_F0}
&\inf_{t\in[0,M_1]}d_F(t)\nonumber\\
&\ge \min\left\{\left(1-F(\tfrac12M_1)\right)\int_{e\ge \tfrac12M_1}e^{-1}\,dF_E(e),F(\tfrac12M_1)\int_{e\ge \tfrac12M_1}e^{-1}\,dF_E(e)\right\}.
\end{align}
So, if $\min\{F(\tfrac12M_1),1-F(\tfrac12M_1)\}\ge c>0$, for all $F\in{\cal F}$ and a constant $c>0$, we get $d_F(t)\ge c'>0$ for another constant $c'>0$ and all $t\in[0,M_1]$ and $F\in{\cal F}$, provided
\begin{align*}
\int_{e\ge \tfrac12M_1}e^{-1}\,dF_E(e)>0,
\end{align*}
which follows from our condition (F1) on $F_E$. 

Since $\e<\tfrac12M_1$, the condition
\begin{align}
\label{lower_bound_d_F}
\inf_{F\in{\cal F}}\min\{F(\tfrac12M_1),1-F(\tfrac12M_1)\}\ge c>0
\end{align}
is implied by  condition (\ref{separation_condition}). This means that we can prove that soluions of the integral equation are uniformly bounded for $F\in{\cal F}$. We get the folloiwing result.

\begin{lemma}
\label{lemma:uniform_boundedness}
Let $F_E$ be a distribution function on $[0,M_2]$, satisfying condition (F1) and let the class of distribution function ${\cal F}$ be defined by (F2).  Let $\psi:[0,M_1]\to\R$ be a bounded right-continuous function with left-limits (cadlag) on $[0,M_1]$.              
Then the solutions $\f_F$ of equation (\ref{phi-int}) are uniformly bounded in $F\in{\cal F}$.
\end{lemma}

\begin{proof}
The equation can be written
\begin{align}
\label{reduced_eq}
&\f(t)\nonumber\\
&=d_F(t)^{-1}\psi(t)+d_F(t)^{-1}F(t)\{1-F(t)\}\nonumber\\
&\qquad\qquad\cdot\int e^{-1}\left[\frac{\f(t+e)-\f(t)}{F(t+e)-F(t)}1_{\{e<M_1-t\}}-\frac{\f(t)-\f(t-e)}{F(t)-F(t-e)}1_{\{e<t\}}\right]\,dF_E(e).
\end{align}
where $d_F$ is defined by (\ref{d_F}). Since $d_F(t)^{-1}\psi(t)$ is clearly uniformly bounded, using (\ref{lower_bound_d_F0}), (\ref{lower_bound_d_F}) and the boundedness of $\psi$, we only have to consider the second term on the right-hand side of (\ref{reduced_eq}). But for this term we use the same type of argument again that we used in the proof of Lemma \ref{lemma:uniqueness_sol_inteq}. If the solution $\f_F$ attains its supremum at $s\in[0,M_1]$, this term is nonpositive at $s$. Otherwise we take the limit from the left. So the upper bound is bounded by
\begin{align*}
d_F(s)^{-1}\psi(s)+d_F(s)^{-1}F(s)\{1-F(s)\}
\end{align*}
or its limit from the left at $s$ if the supremum is not attained.

The same type of argument applies to the lower bound.
\end{proof}

Analogously to the interval censoring model, we have to consider the function $\xi_F$, defined by
\begin{align*}
\xi_F(x)=\frac{\phi_F(x)}{F(x)\{1-F(x)\}},\quad x\in(0,M_1),
\end{align*}
where $\f_F$ solves (\ref{phi-int}) and where we define $0/0=0$. The function $\xi_F$ satisfies a similar integral equation as the function $\f_F$. This is seen in the following way.

We have:
\begin{align*}
\f_F(y)-\f_F(x)&=\left\{\xi_F(y)-\xi_F(x)\right\}F(y)\{1-F(y)\}\\
&\qquad\qquad+\xi_F(x)\bigl[F(y)\{1-F(y)\}-F(x)\{1-F(x)\}\bigr].
\end{align*}
Moreover, if $F(y)>F(x)$,
\begin{align*}
\frac{F(y)\{1-F(y)\}-F(x)\{1-F(x)}{F(y)-F(x)}=1-F(x)-F(y).
\end{align*}
So we find, wrting $\f=\f_F$ and $\xi=\xi_F$:
\begin{align*}
&\int_{e<M_1-t} e^{-1}\frac{\f(t+e)-\f(t)}{F(t+e)-F(t)}\,dF_E(e)\\
&=\int_{e<M_1-t}e^{-1}\frac{\xi(t+e)-\xi(t)}{F(t+e)-F(t)}\,F(t+e)\{1-F(t+e)\}dF_E(e)\\
&\qquad+\xi(t)\int_{e<M_1-t}e^{-1}\,\{1-F(t+e)-F(t)\}\,dF_E(e)
\end{align*}
In a similar way we find:
\begin{align*}
&\int_{e<t} e^{-1}\frac{\f(t)-\f(t-e)}{F(t)-F(t-e)}\,dF_E(e)\\
&=\int_{e<t}e^{-1}\frac{\xi(t)-\xi(t-e)}{F(t)-F(t-e)}\,F(t-e)\{1-F(t-e)\}dF_E(e)\\
&\qquad+\xi(t)\int_{e<t}e^{-1}\,\{1-F(t)-F(t-e)\}\,dF_E(e).
\end{align*}

So we find that $\xi_F$ solves the equation
\begin{align*}
c(t)\xi(t)&=\left\{d_F(t)F(t)\{1-F(t)\right\}^{-1}\psi(t)\\
&\qquad+\int_{e<M_1-t}e^{-1}\frac{\xi(t+e)-\xi(t)}{F(t+e)-F(t)}\,F(t+e)\{1-F(t+e)\}dF_E(e)\\
&\qquad-\int_{e<t}e^{-1}\frac{\xi(t)-\xi(t-e)}{F(t)-F(t-e)}\,F(t-e)\{1-F(t-e)\}dF_E(e)
\end{align*}
in $\xi$, where $c(t)$ is defined by:
\begin{align*}
c(t)&=1+\int_{e<t}e^{-1}\,\{1-F(t)-F(t-e)\}\,dF_E(e)\\
&\qquad-\int_{e<M_1-t}e^{-1}\,\{1-F(t+e)-F(t)\}\,dF_E(e)\\
&=1+\int_{e<t\wedge(M_1-t)}e^{-1}\,\{F(t+e)-F(t-e)\}\,dF_E(e)\\
&\qquad+\int_{t\le e<M_1-t}e^{-1}\,\{1-F(t)\}\,dF_E(e)+\int_{M_1-t\le e<t}e^{-1}\,F(t)\,dF_E(e).
\end{align*}

So the equations have the same properties as the integral equations in the interval censoring, case 2, model and we get in particular the following lemma (Lemma 10.5 on p.\ 297 of \cite{piet_geurt:14}).

\begin{lemma}
\label{lemma:ubder}
Let ${\cal F}$ be the class of distribution functions on $[0,M_1]$ that are either absolutely continuous, with a continuous density staying away from zero on $[0,M_1]$, or a piecewise constant distribution function with a finite number of jumps, satisfying assumption (F2). Let $\psi$ be continuous, with a bounded derivative except at  an at most  countable number of points, where right and left limits exist. Then:
\begin{enumerate}
\item[(i)] The derivative of $\f_F$ at the points of continuity is  bounded, uniformly over $F \in
\cal F$ and the points of continuity, implying
$$
|\f_F(y)-\f_F(x)| \leq K_1\, |y-x|
$$
if $y$ and $x$ are in the same interval between jumps. $K_1$ is independent of $F$ and $x$.
The same holds when $\f_F$ is replaced by $\xi_F$.
\item[(ii)] At the discontinuity points $x$ of $F$,
\begin{align*}
|\f_F(x)-\f_F(x-)| \leq K_2 \, |F(x)-F(x-)|+K_3|k(x)-k(x-)|,
\end{align*}
and
\begin{align*}
\xi_F(x)-\xi_F(x-)| \leq K_2 \, |F(x)-F(x-)|+K_3|k(x)-k(x-)|,
\end{align*}
with $K_2>0$ and $K_3>0$ independent of $x$ and $F$.
\end{enumerate}
\end{lemma}

The proof follows the steps of the proof of Lemma 10.5 on p.\ 297 of \cite{piet_geurt:14} and is therefore omitted. Note however, that we have to extend this lemma, because we want to treat discontinuous functions $\psi$ such as $1_{[0,a]}$ on the right-hand side of the integral equation.

We will need the following extension, analogous to Corollary 4.2 in \cite{piet:96}.

\begin{corollary}
\label{cor:ubder}
Let ${\cal F}$ be the class of distribution functions on $[0,M_1]$ that are either absolutely continuous, with a continuous density staying away from zero on $(0,M_1)$, or a piecewise constant distribution function with a finite number of jumps, satisfying assumption (F2). Let $\psi$ be a bounded right-continuous function with art most a finite set of discontinuities $D$, with a bounded derivative except at  an at most  finite number of points, where right and left limits exist. Then:
\begin{enumerate}
\item[(i)] The derivative of $\f_F$ at the points of continuity is  bounded, uniformly over $F \in
\cal F$ and the points of continuity, implying
$$
|\f_F(y)-\f_F(x)| \leq K_1\, |y-x|
$$
if $y$ and $x$ are in the same interval between jumps. $K_1$ is independent of $F$ and $x$.
The same holds when $\f_F$ is replaced by $\xi_F$.
\item[(ii)] At the discontinuity points $x$ of $F$ or $\psi$:
$$
|\f_F(x)-\f_F(x-)| \leq K_2|F(x)-F(x-)+K_3|\psi(x)-\psi(x-)|
$$
and similarly
$$
|\xi_F(x)-\xi_F(x-)| \leq K_2|F(x)-F(x-)+K_3|\psi(x)-\psi(x-)|
$$
where $K_2$ and $K_3$ are positive constants independent of $x$ and $F$.
\end{enumerate}
\end{corollary}
The proof follows from the preceding results in the same way as the proofs of Lemma 4.3 and Corollary 4.2 in \cite{piet:96} follows from the correponding results on the integral equation there.

\subsection{Proof of Theorem 4}
\label{subsection:proof_theorem4}

A fundamental tool in our proof is the so-called ``switch relation'', see, e.g., Section 3.8 in \cite{piet_geurt:14}.
We define, for $a\in(0,1)$
\begin{align*}
U_n(a)=\text{argmin}\{t\in[0,\infty): V_n(t)-aG_n(t)\},
\end{align*}
where $V_n$ is defined by (3.7) and $G_n = G_{n,\hat F_n}$ is defined by (3.6) for $F=\hat F_n$.
Then we have the so-called {\it switch relation}:
\begin{align*}
\hat F_n(t)\ge a \iff G_n(t)\ge G_n(U_n(a))
\iff t\ge U_n(a).
\end{align*}
see, e.g., (3.35) and Figure 3.7 in Section 3.8 of \cite{piet_geurt:14}.

We have:
\begin{align*}
\P\left\{n^{1/3}\left\{\hat F_n(t_0)-F_0(t_0)\right\}\ge x\right\}
=\P\left\{n^{1/3}\left\{U_n(a_0+n^{-1/3}x)-t_0\right\}\le0\right\},
\end{align*}
where $a_0=F_0(t_0)$. Using the property that the argmin function does not change if we add constants to the object function, we get:
\begin{align*}
&U_n(a_0+n^{-1/3}x)
=\text{argmin}\left\{t\in[0,\infty): V_n(t)-\bigl(a_0+n^{-1/3}x\bigr)G_n(t)\right\}\\
&=\text{argmin}\left\{t\in[0,\infty): V_n(t)-V_n(t_0)-\bigl(a_0+n^{-1/3}x\bigr)\bigl\{G_n(t)-G_n(t_0)\bigr\}\right\}\\
&=\text{argmin}\Biggl\{t_0+n^{-1/3}t\ge0:\int_{u\in(t_0,t_0+n^{-1/3}t]}\hat F_n(u)\,dG_n(u)\\
&\qquad\qquad+W_{n,\hat F_n}(t_0+n^{-1/3}t)
-W_{n,\hat F_n}(t_0)-n^{-1/3}x\left\{G_n(t_0+n^{-1/3}t)-G_n(t_0)\right\}\Biggr\},
\end{align*}
where
\begin{align*}
\int_{u\in(t,v]}\hat F_n(u)\,dG_n(u)
=-\int_{u\in[v,t)}\hat F_n(u)\,dG_n(u),\qquad\text{ if }v<t.
\end{align*}

\begin{remark}
{\rm Note that by the assumptions on $F_E$ and $F_0$ we may assume that $\hat F_n(s)-\hat F_n(s-e)$ stays away from zero for $n$ sufficiently large if $e\ge \e$ and $s-e\ge0$.
}
\end{remark}

Let $m=\max_j(S_j-E_j)_+$, and let $X_n$ be defined as in Lemma \ref{lemma:BM_part} below. Then, letting $\d_1=1_{\{s\le e\}}$, $\d_2=1_{\{e<s\le M_1\}}$ and $\d_3=1-\d_1-\d_2$.
\begin{align*}
&W_{n,\hat F_n}(t_0+n^{-1/3}t)-W_{n,\hat F_n}(t_0)\\
&=X_n(t)+\int_{t_0<s\le t_0+n^{-1/3}t}\frac{\d_1}{\hat F_n(s)}\,dQ_0(e,s)\\
&\qquad\qquad-\int_{t_0<s-e\le t_0+n^{-1/3}t}\frac{\d_2}{\hat F_n(s)-\hat F_n(s-e)}\,dQ_0(e,s)\\
&\qquad\qquad+\int_{t_0<s\le t_0+n^{-1/3}t}\frac{\d_2}{\hat F_n(s)-\hat F_n(s-e)}\,dQ_0(e,s)\\
&\qquad\qquad-\int_{t_0<s-e\le t_0+n^{-1/3}t}\frac{\d_3}{1-\hat F_n(s-e)}\,dQ_0(e,s)\\
&=X_n(t)+\int_{t_0<s\le (t_0+n^{-1/3}t)\wedge m,\,s\le e}e^{-1}\frac{F_0(s)}{\hat F_n(s)}\,ds\,dF_E(e)\\
&\qquad\qquad-\int_{t_0<s-e\le t_0+n^{-1/3}t,\,e<s\le m}e^{-1}\frac{F_0(s)-F_0(s-e)}{\hat F_n(s)-\hat F_n(s-e)}\,ds\,dF_E(e)\\
&\qquad\qquad+\int_{t_0<s\le t_0+n^{-1/3}t,\,e<s\le m}e^{-1}\frac{F_0(s)-F_0(s-e)}{\hat F_n(s)-\hat F_n(s-e)}\,ds\,dF_E(e)\\
&\qquad\qquad-\int_{t_0<s-e\le t_0+n^{-1/3}t,\,s> m}e^{-1}\frac{1-F_0(s-e)}{1-\hat F_n(s-e)}\,ds\,dF_E(e).
\end{align*}
By a change of variables, we can write the last expression in the form:
\begin{align*}
&X_n(t)+\int_{t_0<s\le (t_0+n^{-1/3}t)\wedge m}e^{-1}\frac{F_0(s)}{\hat F_n(s)}\,ds\,dF_E(e)\\
&\qquad\qquad-\int_{t_0<s\le t_0+n^{-1/3}t,\,s+e\le m}e^{-1}\frac{F_0(s+e)-F_0(s)}{\hat F_n(s+e)-\hat F_n(s)}\,ds\,dF_E(e)\\
&\qquad\qquad+\int_{t_0<s\le t_0+n^{-1/3}t,\,e<s\le m}e^{-1}\frac{F_0(s)-F_0(s-e)}{\hat F_n(s)-\hat F_n(s-e)}\,ds\,dF_E(e)\\
&\qquad\qquad-\int_{t_0<s\le t_0+n^{-1/3}t,\,s+e> m}e^{-1}\frac{1-F_0(s)}{1-\hat F_n(s)}\,ds\,dF_E(e),
\end{align*}
For future reference, we define
\begin{align}
\label{process_Y_n}
&Y_n(t)=\int_{t_0<s\le (t_0+n^{-1/3}t)\wedge m}e^{-1}\frac{F_0(s)}{\hat F_n(s)}\,ds\,dF_E(e)\nonumber\\
&\qquad\qquad-\int_{t_0<s\le t_0+n^{-1/3}t,\,s+e\le m}e^{-1}\frac{F_0(s+e)-F_0(s)}{\hat F_n(s+e)-\hat F_n(s)}\,ds\,dF_E(e)\nonumber\\
&\qquad\qquad+\int_{t_0<s\le t_0+n^{-1/3}t,\,e<s\le m}e^{-1}\frac{F_0(s)-F_0(s-e)}{\hat F_n(s)-\hat F_n(s-e)}\,ds\,dF_E(e)\nonumber\\
&\qquad\qquad-\int_{t_0<s\le t_0+n^{-1/3}t,\,s+e> m}e^{-1}\frac{1-F_0(s)}{1-\hat F_n(s)}\,ds\,dF_E(e),
\end{align}

We have the following lemma.
\begin{lemma}
\label{lemma:BM_part}
Let, under the conditions of Theorem 4, the process $X_n$ be defined by:
\begin{align*}
X_n(t)&=\int_{t_0<s\le t_0+n^{-1/3}t}\frac{\d_1}{\hat F_n(s)}\,d(\Q_n-Q_0)(e,s)\\
&\qquad\qquad-\int_{t_0<s-e\le t_0+n^{-1/3}t}\frac{\d_2}{\hat F_n(s)-\hat F_n(s-e)}\,d(\Q_n-Q_0)(e,s)\\
&\qquad\qquad+\int_{t_0<s\le t_0+n^{-1/3}t}\frac{\d_2}{\hat F_n(s)-\hat F_n(s-e)}\,d(\Q_n-Q_0)(e,s)\\
&\qquad\qquad-\int_{t_0<s-e\le t_0+n^{-1/3}t}\frac{\d_3}{1-\hat F_n(s-e)}\,d(\Q_n-Q_0)(e,s),
\end{align*}
where $\d_1=1_{\{s\le e\}}$, $\d_2=1_{\{e<s\le M_1\}}$ and $\d_3=1-\d_1-\d_2$.
Let $t_0$ be an interior point of the support of $f_0$.  Then $n^{2/3}X_n$ converges in distribution, in the Skorohod topology, to the process
\begin{align*}
t\mapsto \sqrt{c_E}W(t),\qquad t\in\R,
\end{align*}
where $c_E$ and $W$ are the same as in Theorem 4.
\end{lemma}

\begin{proof}[Proof of Lemma \ref{lemma:BM_part}]
This follows from the convergence to the same limit process of
\begin{align}
\label{X_n_F_0}
X_n(t)&=\int_{t_0<s\le t_0+n^{-1/3}t}\frac{\d_1}{F_0(s)}\,d(\Q_n-Q_0)(e,s)\\
&\qquad\qquad-\int_{t_0<s-e\le t_0+n^{-1/3}t}\frac{\d_2}{F_0(s)-F_0(s-e)}\,d(\Q_n-Q_0)(e,s)\\
&\qquad\qquad+\int_{t_0<s\le t_0+n^{-1/3}t}\frac{\d_2}{F_0(0)-F_0(s-e)}\,d(\Q_n-Q_0)(e,s)\\
&\qquad\qquad-\int_{t_0<s-e\le t_0+n^{-1/3}t}\frac{\d_3}{1-F_0(s-e)}\,d(\Q_n-Q_0)(e,s),
\end{align}
and the consistency of $\hat F_n$ together with the entropy with bracketing for the $L_2$-norm of the functions
\begin{align*}
(e,s)\mapsto\frac1{F(s)-F(s-e)},\qquad s\in[t_0-n^{-1/3}M,t_0+n^{-1/3}M],\qquad e\ge\e,
\end{align*}
for $M>0$ and distribution functions $F$ such that $\{F(s)-F(s-e)\}1_{\{e\ge\e\}}$ stays away from zero for $s$ in the relevant interval (see, e.g., p.\  59 of \cite{piet_geurt:14}).

Note that we get for the variance of (\ref{X_n_F_0}):
\begin{align*}
&\text{var}\bigl(X_n(t)\bigr)\\
&\sim n^{-1}\int_{t_0<s\le t_0+n^{-1/3}t}\frac{\d_1}{F_0(s)^2}\,dQ_0(e,s)\\
&\qquad\qquad+n^{-1}\int_{t_0-e<s\le t_0-e+n^{-1/3}t}\frac{\d_2}{\bigl\{F_0(s)-F_0(s-e)\bigr\}^2}\,dQ_0(e,s)\\
&\qquad\qquad+n^{-1}\int_{t_0<s\le t_0+n^{-1/3}t}\frac{\d_2}{\bigl\{F_0(s)-F_0(s-e)\bigr\}^2}\,dQ_0(e,s)\\
&\qquad\qquad+n^{-1}\int_{t_0-e<s\le t_0-e+n^{-1/3}t}\frac{\d_3}{\bigl\{1-F_0(s-e)\bigr\}^2}\,dQ_0(e,s)\\
&= n^{-1}\int_{t_0<s\le t_0+n^{-1/3}t,\,s\le M_1\wedge e}e^{-1}\frac{1}{F_0(s)}\,ds\,dF_E(e)\\
&\qquad\qquad+n^{-1}\int_{t_0<s\le t_0+n^{-1/3}t,\,s+e\le M_1}e^{-1}\frac1{F_0(s+e)-F_0(s)}\,ds\,dF_E(e)\\
&\qquad\qquad+n^{-1}\int_{t_0<s\le t_0+n^{-1/3}t,\,e<s\le M_1}e^{-1}\frac1{F_0(s)-F_0(s-e)}\,ds\,dF_E(e)\\\
&\qquad\qquad+n^{-1}\int_{t_0<s\le t_0+n^{-1/3}t,\,s+e>M_1}e^{-1}\frac1{1-F_0(s)}\,ds\,dF_E(e)\\
&\sim n^{-4/3}t\int e^{-1}\frac1{F_0(t_0)-F_0(t_0-e)}\,dF_E(e)+n^{-4/3}t\int e^{-1}\frac1{F_0(t_0+e)-F_0(t_0)}\,dF_E(e).
\end{align*}
\end{proof}

We will also need the following rate result for the $L_2$-distance.

\begin{lemma}
\label{lemma:L_2-bound}
Let the conditions of Theorem 4 be satisfied.. Then
\begin{align}
\label{L_2-bound}
\|\hat F_n-F_0\|=O_p\left(n^{-1/3}\right).
\end{align}
\end{lemma}

\begin{proof}
We define the (convolution) density $q_F$ by (\ref{convolution}) and follow the exposition in \cite{geer:00}, Example 7.4.4. The condition that the exposure time $E$ stays away from zero is comparable to the condition (7.41) on p. 116 of \cite{geer:00} that the intervals $[U,V]$ have a length which stays away from zero. In this case we find that the squared Hellinger distance satisfies:
\begin{align*}
&h\left(q_{\hat F_n},q_{F_0}\right)^2=\tfrac12\int\left\{\sqrt{q_{\hat F_n}}-\sqrt{q_{F_0}}\right\}^2\,ds\,dF_E(e)\\
&=\tfrac12\int_{e\in[\e,M_2]} e^{-1}\int_{s\in[0,M_1+M_2]}\left\{\sqrt{\hat F_n(s)-\hat F_n(s-e)}-\sqrt{F_0(s)-F_0(s-e)}\right\}^2\,ds\,dF_E(e)\\
&=O_p\left(n^{-2/3}\right),
\end{align*}
see (7.42) in \cite{geer:00}.

We can write the next to last term above in the form (using Fubini's theorem and a change of variables):
\begin{align*}
&\tfrac12\int_{s\le M_1+M_2} \left\{\sqrt{\hat F_n(s)}-\sqrt{F_0(s)}\right\}^2\int_{e\ge s} e^{-1}\,dF_E(e)\,ds\\
&+\tfrac12\int_{s\le M_1+M_2} \left\{\sqrt{1-\hat F_n(s)}-\sqrt{1-F_0(s)}\right\}^2\int_{e\ge M_1-s} e^{-1}\,dF_E(e)\,ds\\
&+\tfrac12\int_{s\le M_1+M_2} \left\{\sqrt{\hat F_n(s+e)-\hat F_n(s)}-\sqrt{F_0(s+e)-F_0(s)}\right\}^2\int_{e<M_1-s} e^{-1}\,dF_E(e)\,ds\\
&+\tfrac12\int_{s\le M_1+M_2} \left\{\sqrt{\hat F_n(s)-\hat F_n(s-e)}-\sqrt{F_0(s)-F_0(s-e)}\right\}^2\int_{e<t} e^{-1}\,dF_E(e)\,ds.
\end{align*}
So we get in particular
\begin{align*}
&\tfrac12\int_{s\le M_1+M_2} \left\{\sqrt{\hat F_n(s)}-\sqrt{F_0(s)}\right\}^2\int_{e\ge s} e^{-1}\,dF_E(e)\,ds\\
&\qquad+\tfrac12\int_{s\le M_1+M_2} \left\{\sqrt{1-\hat F_n(s)}-\sqrt{1-F_0(s)}\right\}^2\int_{e\ge M_1-s} e^{-1}\,dF_E(e)\,ds\\
&=O_p\left(n^{-2/3}\right),
\end{align*}
implying
\begin{align*}
&\tfrac12\int_{s\le M_1+M_2} \left\{\sqrt{\hat F_n(s)}-\sqrt{F_0(s)}\right\}^2\,ds
+\tfrac12\int_{s\le M_1+M_2} \left\{\sqrt{1-\hat F_n(s)}-\sqrt{1-F_0(s)}\right\}^2\,ds\\
&=O_p\left(n^{-2/3}\right),
\end{align*}
by our assumptions on $F_E$. The relation
\begin{align*}
(a-b)^2=\left\{\sqrt{a}-\sqrt{b}\right\}^2\left\{\sqrt{a}+\sqrt{b}\right\}^2\le2\left\{\sqrt{a}-\sqrt{b}\right\}^2\,\qquad ,\,a,b\in[0,1],
\end{align*}
now gives the result.
\end{proof}

The following lemma. corresponding to Lemma 4,4 on p.\ 146 of \cite{piet:96} is crucial in our proof.

\begin{lemma}
\label{lemma:supremum_lemma}
Let the conditions of Theorem 4 be satisfied and let $\hat F_n$ be the MLE. Then:
\begin{align*}
\sup_{t\in[0,M_1]}\sqrt{n}\int_0^t\bigl\{\hat F_n(x)-F_0(x)\bigr\}\,dx=O_p\left(1\right).
\end{align*}
\end{lemma}
\begin{proof}
Let $\f_{t,\hat F_n}$ be the solution of the integral eqaution
\begin{align*}
\int e^{-1}\left[\frac{\f(x+e)-\f(x)}{\hat F_n(x+e)-\hat F_n(x)}-\frac{\f(x)-\f(x-e)}{\hat F_n(x)-\hat F_n(x-e)}\right]\,dF_E(e)
=1_{[0,t)}(x),\qquad x\in[0,M_1],
\end{align*}
see (\ref{phi-int}) above, and let $\th_{t,\hat F_n}$ be defined by
\begin{align*}
\th_{t,\hat F_n}(e,s,\d_1,\d_2)&=\d_1\frac{\f_{t,\hat F_n}(s)}{\hat F_n(s)}-\d_3\frac{\f_{t,\hat F_n}(s-e)}{1-\hat F_n(s-e)}
+\d_2\frac{\f(s)-\f(s-e)}{\hat F_n(s)-\hat F_n(s-e)}\,,
\end{align*}
where $\d_1=1_{\{s\le e\}}$, $\d_2=1_{\{e<s\le M_1\}}$ and $\d_3=1-\d_1-\d_2$. A picture of such a $\f_{t,\hat F_n}$ is given in Figure \ref{figure:phi_MLE}.

\begin{figure}[!ht]
\centering
\includegraphics[width=0.5\textwidth]{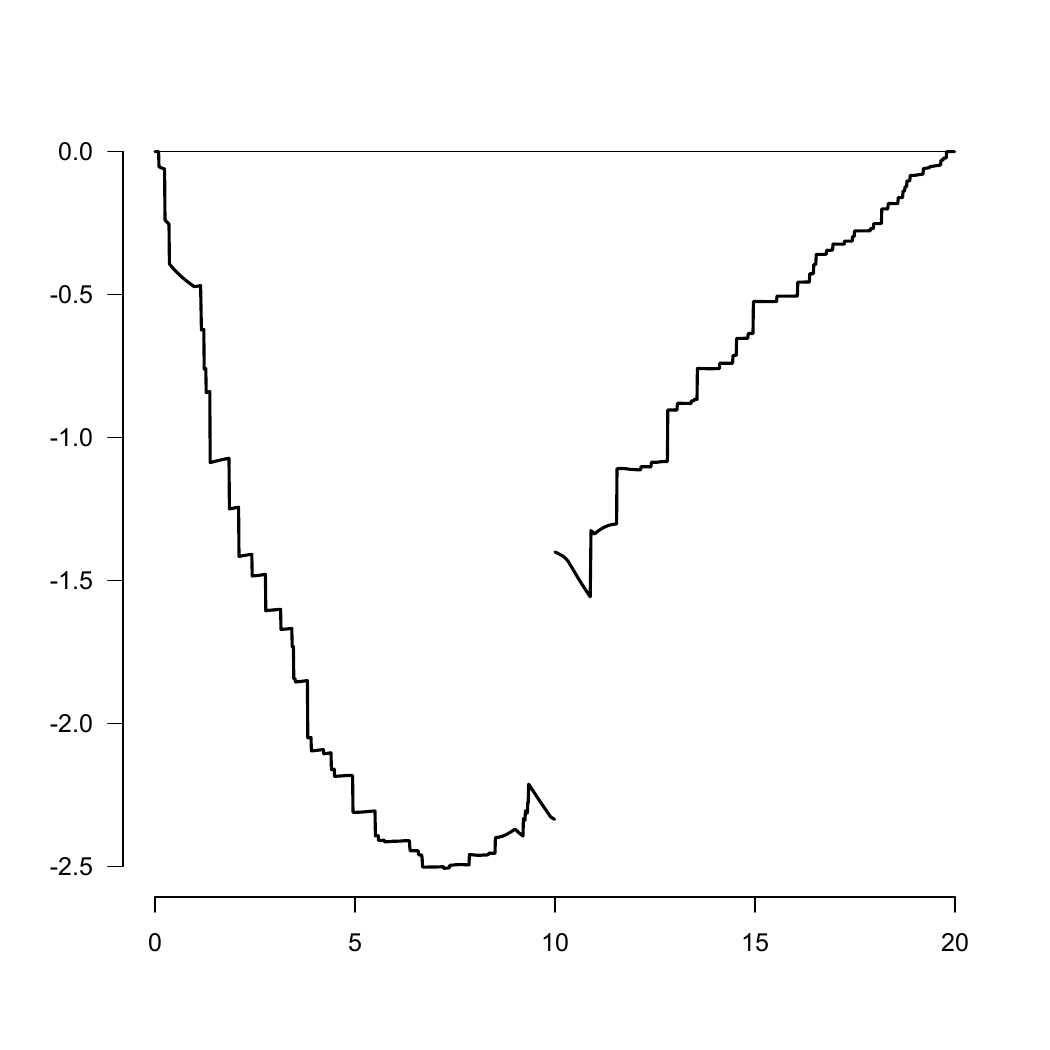}
\caption{The solution $\f_{t,\hat F_n}$for $t=10$ and $n=10^4$, if $F_0$ is uniform on $[0,20]$, and $F_E$ is uniform on $[1,30]$.}
\label{figure:phi_MLE}
\end{figure}

Note the similarity of Figure \ref{figure:phi_MLE} to Figure \ref{figure:phi_discont} above, where $F_0$ instead of $\hat F_n$ is used.
We have, if $\f=\f_{t,\hat F_n}$,
\begin{align*}
&\int_0^t\bigl\{\hat F_n(x)-F_0(x)\bigr\}\,dx=\int_{x=0}^{M_1}1_{[0,t)}(x)\bigl\{\hat F_n(x)-F_0(x)\bigr\}\,dx\\
&=\int_{x=0}^{M_1}\bigl\{\hat F_n(x)-F_0(x)\bigr\}\int_{e=0}^{M_2} e^{-1}\left[\frac{\f(x+e)-\f(x)}{\hat F_n(x+e)-\hat F_n(x)}\right.\\
&\qquad\qquad\qquad\qquad\qquad\qquad\qquad\qquad\qquad\qquad\left.-\frac{\f(x)-\f(x-e)}{\hat F_n(x)-\hat F_n(x-e)}\right]\,dF_E(e)\,dx\\
&=-\int_{e=0}^{M_2}e^{-1}\int_{x=e}^{M_1}\frac{\{F_0(x)-F_0(x-e)\}\{\f(x)-\f(x-e)\}}{\hat F_n(x)-\hat F_n(x-e)}\,dx\,dF_E(e)\\
&\qquad\qquad-\int_{e=0}^{M_2}e^{-1}\int_{x=0}^{e}\frac{F_0(x)\f(x)}{\hat F_n(x)}\,dx\,dF_E(e)\\
&\qquad\qquad\qquad\qquad\qquad\qquad+\int_{e=0}^{M_2}e^{-1}\int_{x=M_1}^{M_1+e}\frac{\{1-F_0(x-e)\}\f(x-e)}{1-\hat F_n(x-e)}\,dx\,dF_E(e)\\
&=-\int \th_{t,\hat F_n}(e,s,\d_1,\d_2)\,dQ_0(e,s).
\end{align*}
Note that
\begin{align*}
&\int_{e=0}^{M_2}\hat F_n(x)\int_{x=0}^{M_1} e^{-1}\left[\frac{\f(x+e)-\f(x)}{\hat F_n(x+e)-\hat F_n(x)}
-\frac{\f(x)-\f(x-e)}{\hat F_n(x)-\hat F_n(x-e)}\right]\,dF_E(e)\,dx\\
&=\int_{e=0}^{M_2}e^{-1}\int_{x=e}^{M_1}\{\f(x)-\f(x-e)\}\,dx\,dF_E(e)+\int_{e=0}^{M_2}e^{-1}\int_{x=0}^{e}\f(x)\,dx\,dF_E(e)\\
&\qquad\qquad\qquad\qquad\qquad\qquad-\int_{e=0}^{M_2}e^{-1}\int_{x=M_1}^{M_1+e}\f(x-e)\,dx\,dF_E(e)\\
&=0.
\end{align*}
Also note that for the argument $x+e>M_1$ we use
\begin{align*}
\hat F_n-F_0=1-F_0-\{1-\hat F_n\}.
\end{align*}
The argument now follows the reasoning of the proof of Lemma 4.4 in \cite{piet:96}, where we use the properties of the solution of the integral equations, discussed in Section \ref{subsection:Integral_equations} above and Lemma \ref{lemma:L_2-bound}. For example, we change the function $\f_{t,\hat F_n}$ to a piecewise constant version $\bar\f_{t,\hat F_n}$, piecewise constant on the same intervals as $\hat F_n$,
except possibly the interval containing $t$, for example using (4.37) on p.\ 146 of \cite{piet:96},
and define
\begin{align*}
\bar\th_{t,\hat F_n}(e,s,\d_1,\d_2)&=\d_1\frac{\bar\f_{t,\hat F_n}(s)}{\hat F_n(s)}-\d_3\frac{\bar\f_{t,\hat F_n}(s-e)}{1-\hat F_n(s-e)}
+\d_2\frac{\bar\f(s)-\bar\f(s-e)}{\hat F_n(s)-\hat F_n(s-e)}\,,
\end{align*}
where $\d_1=1_{\{s\le e\}}$, $\d_2=1_{\{e<s\le M_1\}}$ and $\d_3=1-\d_1-\d_2$.
Then, using Lemma \ref{lemma:L_2-bound} above, we find:
\begin{align*}
\sup_{t\in[0,M_1]}\left|\int\bigl\{\bar\th_{t,\hat F_n}-\th_{t,\hat F_n}\bigr\}\,dQ_0\right|=O_p\left(n^{-2/3}\right).
\end{align*}
Next we get:
\begin{align*}
\sup_{t\in[0,M_1]}\left|\int\bar\th_{t,\hat F_n}\,dQ_0\right|
=\sup_{t\in[0,M_1]}\left|\int\bar\th_{t,\hat F_n}\,d\left(Q_0-\Q_n\right)\right|+O_p\left(n^{-2/3}\right)
=O_p\left(n^{-1/2}\right)
\end{align*}
which gives the desired result.
\end{proof}

\begin{corollary}
\label{extension_corollary}
Let the conditions of Theorem 4 be satisfied and let $\hat F_n$ be the MLE. Moreover, let ${\cal G}$ be a set of right-continuous function with left limits $g:[0,M_1]\to\R$ which are of uniformly bounded variation. Then:
\begin{align*}
\sup_{g\in{\cal G}}\sqrt{n}\left|\int_0^M g(x)\bigl\{\hat F_n(x)-F_0(x)\bigr\}\,dx\right|=O_p(1).
\end{align*}
\end{corollary}

\noindent
The proof of this corollary follows in the same way as the proof of Corollary 4.3 in \cite{piet:96}.

The following rough upper bound will also be useful.

\begin{lemma}
\label{lemma:rough_upper_bound}
Let the conditions of Theorem 4 be satisfied. Then
\begin{align*}
\sup_{x\in[0,M_1]}\left|\hat F_n(x)-F_0(x)\right|=O_p\left(n^{-1/4}\right).
\end{align*}
\end{lemma}

\begin{proof}
The proof is analogous to the proof of Corollary 4.4 in \cite{piet:96}.
\end{proof}

\begin{lemma}
\label{lemma:decomp}
Let the conditions of Theorem 4 be satisfied, and let the proces $Y_n$ be defined by (\ref{process_Y_n}). Then, for arbitrary $M>0$ and $t\in[-M,M]$:
\begin{align}
\label{basic_expansion}
\int_{s\in(t_0,t_0+n^{-1/3}t]}\bigl\{\hat F_n(s)-F_0(t_0)\bigr\}\,dG_n(s)+Y_n(t)
&=\tfrac12f_0(t_0)c_En^{-2/3}t^2+B_n(t)\nonumber\\
&\qquad\qquad\qquad+o_p\left(n^{-2/3}\right),
\end{align}
where
\begin{align}
\label{def_B_n(t)}
B_n(t)&=\int e^{-1}\int_{s\in[t_0,t_0+n^{-1/3}t)}\left\{\frac{\hat F_n(s-e)-F_0(s-e)}{\hat F_n(s)-\hat F_n(s-e)}\right.\\
&\left.\qquad\qquad\qquad\qquad\qquad\qquad\qquad\qquad+\frac{\hat F_n(s+e)-F_0(s+e)}{\hat F_n(s+e)-\hat F_n(s)}\right\}\,ds\,dF_E(e).\nonumber
\end{align}
\end{lemma}
\begin{proof}
Let $Y_n$ be defined by (\ref{process_Y_n}). We can write:
\begin{align}
\label{def_Y_n2}
Y_n(t)=\int_{t_0<s\le t_0+n^{-1/3}t}e^{-1}\left\{\frac{F_0(s)-F_0(s-e)}{\hat F_n(s)-\hat F_n(s-e)}
-\frac{F_0(s+e)-F_0(s)}{\hat F_n(s+e)-\hat F_n(s)}\right\}\,ds\,dF_E(e),
\end{align}
where $F_0(s-e)=0$ and $F_0(s+e)=1$ can occur.
The last expression can be rewritten in the form $A_n(t)+B_n(t)$, where
\begin{align*}
&A_n(t)\\
&=-\int_{s\in[t_0,t_0+n^{-1/3}t)} e^{-1}\bigl\{\hat F_n(s)-F_0(s)\bigr\}\left\{\frac{1}{\hat F_n(s)-\hat F_n(s-e)}+\frac{1}{\hat F_n(s+e)-\hat F_n(s)}\right\}\\
&\qquad\qquad\qquad\qquad\qquad\qquad\qquad\qquad\qquad\qquad\qquad\qquad\qquad\qquad\qquad\qquad\,ds\,dF_E(e),
\end{align*}
and
\begin{align*}
&B_n(t)\\
&=\int e^{-1}\int_{s\in[t_0,t_0+n^{-1/3}t)}\left\{\frac{\hat F_n(s-e)-F_0(s-e)}{\hat F_n(s)-\hat F_n(s-e)}
+\frac{\hat F_n(s+e)-F_0(s+e)}{\hat F_n(s+e)-\hat F_n(s)}\right\}\,ds\,dF_E(e).
\end{align*}

We also have:
\begin{align*}
&\int_{s\in(t_0,t_0+n^{-1/3}t]}\bigl\{\hat F_n(s)-F_0(t_0)\bigr\}\,dG_n(s)\\
&=\int_{s\in(t_0,t_0+n^{-1/3}t]}\hat F_n(s)\,dG_n(s)
-F_0(t_0)\left\{G_n(t_0+n^{-1/3}t)-G_n(t_0)\right\}\\
&=\int_{s\in(t_0,t_0+n^{-1/3}t]}\left\{\frac{\hat F_n(s)-F_0(t_0)}{\{\hat F_n(s)-\hat F_n(s-e)\}^2}
+\frac{\hat F_n(s)-F_0(t_0)}{\{\hat F_n(s+e)-\hat F_n(s)\}^2}\right\}\,d\Q_n\\
&=\int_{s\in(t_0,t_0+n^{-1/3}t]}\left\{\frac{\hat F_n(s)-F_0(s)}{\{\hat F_n(s)-\hat F_n(s-e)\}^2}
+\frac{\hat F_n(s)-F_0(s)}{\{\hat F_n(s+e)-\hat F_n(s)\}^2}\right\}\,d\Q_n\\
&\qquad+\int_{s\in(t_0,t_0+n^{-1/3}t]}\left\{\frac{F_0(s)-F_0(t_0)}{\{\hat F_n(s)-\hat F_n(s-e)\}^2}
+\frac{F_0(s)-F_0(t_0)}{\{\hat F_n(s+e)-\hat F_n(s)\}^2}\right\}\,d\Q_n\\
&=\int_{s\in(t_0,t_0+n^{-1/3}t]}\left\{\frac{\hat F_n(s)-F_0(s)}{\{\hat F_n(s)-\hat F_n(s-e)\}^2}
+\frac{\hat F_n(s)-F_0(s)}{\{\hat F_n(s+e)-\hat F_n(s)\}^2}\right\}\,d\Q_n\\
&\qquad+\tfrac12c_Ef_0(t_0)t^2+o_p\left(n^{-2/3}\right)\\
&=-A_n(t)+\tfrac12c_Ef_0(t_0)n^{-2/3}t^2+o_p\left(n^{-2/3}\right),
\end{align*}
where $c_E$ is given by (5.2) in Theorem 4. The last equality is seen by writing
\begin{align*}
&\int_{s\in(t_0,t_0+n^{-1/3}t]}\left\{\frac{\hat F_n(s)-F_0(s)}{\{\hat F_n(s)-\hat F_n(s-e)\}^2}
+\frac{\hat F_n(s)-F_0(s)}{\{\hat F_n(s+e)-\hat F_n(s)\}^2}\right\}\,d\Q_n\\
&=\int_{s\in(t_0,t_0+n^{-1/3}t]}\left\{\frac{\hat F_n(s)-F_0(s)}{\{\hat F_n(s)-\hat F_n(s-e)\}^2}
+\frac{\hat F_n(s)-F_0(s)}{\{\hat F_n(s+e)-\hat F_n(s)\}^2}\right\}\,d\bigl(\Q_n-Q_0\bigr)\\
&\qquad+\int_{s\in(t_0,t_0+n^{-1/3}t]}\left\{\frac{\hat F_n(s)-F_0(s)}{\{\hat F_n(s)-\hat F_n(s-e)\}^2}
+\frac{\hat F_n(s)-F_0(s)}{\{\hat F_n(s+e)-\hat F_n(s)\}^2}\right\}\,dQ_0\\
&=-A_n(t)+o_p\left(n^{-2/3}\right),
\end{align*}
using the consistency of $\hat F_n$ to show that the first term after the next to last equality is of order $o_p\left(n^{-2/3}\right)$ and
\begin{align*}
&\int_{s\in(t_0,t_0+n^{-1/3}t]}\left\{\frac{\hat F_n(s)-F_0(s)}{\{\hat F_n(s)-\hat F_n(s-e)\}^2}
+\frac{\hat F_n(s)-F_0(s)}{\{\hat F_n(s+e)-\hat F_n(s)\}^2}\right\}\,dQ_0\\
&=\int_{s\in(t_0,t_0+n^{-1/3}t]}e^{-1}\left\{\frac{\{\hat F_n(s)-F_0(s)\}\{F_0(s)-F_0(s-e)\}}{\{\hat F_n(s)-\hat F_n(s-e)\}^2}\right.\\
&\qquad\qquad\qquad\qquad\qquad\qquad\left.+\frac{\{\hat F_n(s)-F_0(s)\}\{F_0(s+e)-F_0(s)\}}{\{\hat F_n(s+e)-\hat F_n(s)\}^2}\right\}\,ds\,dF_E(e)\\
&=\int_{s\in(t_0,t_0+n^{-1/3}t]}e^{-1}\left\{\frac{\hat F_n(s)-F_0(s)}{\hat F_n(s)-\hat F_n(s-e)}+\frac{\hat F_n(s)-F_0(s)}{\hat F_n(s+e)-\hat F_n(s)}\right\}\,ds\,dF_E(e)\\
&\qquad\qquad\qquad\qquad\qquad\qquad\qquad\qquad\qquad\qquad\qquad\qquad\qquad\qquad+O_p\left(n^{-5/6}\right),
\end{align*}
using Lemma \ref{lemma:rough_upper_bound} in the last step.
So the term $A_n(t)$ drops out and the result now follows.
\end{proof}

The term $B_n(t)$ in Lemma \ref{lemma:decomp} is now treated by using differentiable functional theory. We really have to use smooth functional theory here and cannot use simple $L_2$-bounds or other tools of that type only to show that $B_n(t)$ is of order $o_p(n^{-2/3})$. One could say that this is the heart of the difficulty of the proof. We'll use the following lemma.

\begin{lemma}
\label{lemma:B_n(t)}
Let the conditions of Theorem 4 be satisfied, and let $B_n(t)$ be defined by (\ref{def_B_n(t)}) in Lemma \ref{lemma:decomp}.  Then, for arbitrary $M>0$ and $t\in[-M,M]$:
\begin{align*}
B_n(t)=O_p\left(n^{-5/6}\right).
\end{align*}
\end{lemma}

\begin{proof}
Using Lemma \ref{lemma:rough_upper_bound} again, is is sufficient to show $\tilde B_n(t)=O_p(n^{-5/6})$,
where $\tilde B_n(t)$ is defined by
\begin{align*}
&\tilde B_n(t)\\
&=\int_{s\in[t_0,t_0+n^{-1/3}t)}e^{-1}\left\{\frac{\hat F_n(s-e)-F_0(s-e)}{F_0(s)-F_0(s-e)}
+\frac{\hat F_n(s+e)-F_0(s+e)}{F_0(s+e)-F_0(s)}\right\}\,dF_E(e)\,ds.
\end{align*}
Let $t_0>\e$, where $\e>0$ is defined as in the conditions of Theorem 4. We can write
\begin{align*}
&\int e^{-1}\frac{\hat F_n(s-e)-F_0(s-e)}{F_0(s)-F_0(s-e)}\,dF_E(e)\\
&=\int_{u<s-\e} (s-u)^{-1}\frac{f_E(s-u)}{F_0(s)-F_0(u)}\,\bigl\{\hat F_n(u)-F_0(u)\bigr\}\,du=O_p\left(n^{-1/2}\right),
\end{align*}
uniformly for $s\in[t_0,t_0+n^{-1/3}t)$ by Corollary \ref{extension_corollary} and the continuity of the function
\begin{align*}
s\mapsto (s-u)^{-1}\frac{f_E(s-u)}{F_0(s)-F_0(u)}
\end{align*}
for $s\in[t_0,t_0+n^{-1/3}t)$, if $u$ stays away from $s$.

So:
\begin{align*}
\int_{s\in[t_0,t_0+n^{-1/3}t)}e^{-1}\frac{\hat F_n(s-e)-F_0(s-e)}{F_0(s)-F_0(s-e)}\,dF_E(e)\,ds=O_p\left(n^{-5/6}\right).
\end{align*}

If $t_0\le\e$, the integration interval for $e$ is either empty or of order $n^{-1//3}$, which yields, using $\sup_x|\hat F_n(x)-F_0(x)|=O_p(n^{-1/4})$, in which case:
\begin{align*}
\int_{s\in[t_0,t_0+n^{-1/3}t)}e^{-1}\frac{\hat F_n(s-e)-F_0(s-e)}{F_0(s)-F_0(s-e)}\,dF_E(e)\,ds
=O_p\left(n^{-11/12}\right)=O_p\left(n^{-5/6}\right).
\end{align*}

Similarly,
\begin{align*}
\int_{s\in[t_0,t_0+n^{-1/3}t)}e^{-1}\frac{\hat F_n(s+e)-F_0(s+e)}{F_0(s+e)-F_0(s)}\,dF_E(e)\,ds=O_p\left(n^{-5/6}\right).
\end{align*}
\end{proof}

We finally need the following ``tightness'' lemma, which follows from the negligibility of $B_n(t)$ in (\ref{basic_expansion}) of Lemma \ref{lemma:decomp}, which, in turn, follows from Lemma \ref{lemma:B_n(t)}.

\begin{lemma}
\label{tightness_lemma}
Let the conditions of Theorem 4 be satisfied and let $a_0\in(0,1)$.
Then, for each $\d>0$ and $K_1>0$ a $K_2>0$ can be found such that
\begin{align*}
\P\left\{\sup_{x\in[-K_1,K_1]}n^{1/3}\left\{U_n\bigl(a_0+n^{-1/3}x\bigr)-t_0\right\}>K_2\right\}<\d,
\end{align*}
and
\begin{align*}
\P\left\{\inf_{x\in[-K_1,K_1]}n^{1/3}\left\{U_n\bigl(a_0+n^{-1/3}x\bigr)-t_0\right\}<-K_2\right\}<\e,
\end{align*}
for all large $n$.
\end{lemma}

We now have:
\begin{align*}
&n^{1/3}\left\{U_n(a_0+n^{-1/3}x)-t_0\right\}\\
&=\text{argmin}\Biggl\{t\ge -n^{2/3}t_0:n^{2/3}X_n(t)+\tfrac12c_E f_0(t_0)t^2-n^{1/3}x\left\{G_n(t_0+n^{-1/3}t)-G_n(t_0)\right\}+o_p(1)\Biggr\}\\
&=\text{argmin}\Biggl\{t\ge -n^{2/3}t_0:n^{2/3}X_n(t)+\tfrac12c_E f_0(t_0)t^2-c_Ext+o_p(1)\Biggr\},
\end{align*}
which converges in distribution to the argmin of the process
\begin{align*}
t\mapsto \sqrt{c_E}\,W(t)+\tfrac12c_E f_0(t_0)t^2-c_Ext,\qquad t\in\R,
\end{align*}
where $W$ is two-sided Brownian motion, originating from zero. Theorem 4 now follows from Brownian scaling.

\subsection{Proof of Theorem \ref{th:limit_SMLE}}
\label{subsection:proof_Theorem5.1}
\begin{proof}
This time the adjoint equation (see Section \ref{subsection:score_operators} of this appendix) is (for $F=\hat F_n$):
\begin{align}
\label{phi-eq_SMLE}
\left[A^*b\right](x)&=
\int_{e>0}e^{-1}\int_{s\in(x,x+e)}\frac{\f(s)-\f(s-e)}{\hat F_n(s)-\hat F_n(s-e)}\,ds\,dF_E(e)\nonumber\\
&=\IK((t-x)/h_n)-\int \IK_h((t-y)/h_n)\,d\hat F_n(y),\qquad x\in(0,M_1).
\end{align}
Differentiating the equation w.r.t.\ $x$ we get, letting $h=h_n$:
\begin{align}
\label{inteq_SMLE}
&\int_{e>0}e^{-1}\left\{\frac{\f(x+e)-\f(x)}{\hat F_n(x+e)-\hat F_n(x)}
-\frac{\f(x)-\f(x-e)}{\hat F_n(x)-\hat F_n(x-e)}\right\}\,dF_E(e)=-K_h(t-x).
\end{align}

So we get, using integration by parts, if $\f=\f_{t,\hat F_n}$ solves (\ref{inteq_SMLE}) for $F=\hat F_n$:
\begin{align*}
&\int \IK_h(t-x)\,d\bigl(\hat F_n-F_0\bigr)(x)
=\int\bigl(\hat F_n-F_0\bigr)(x)K_h(t-x)\,dx\\
&=\int\bigl(\hat F_n-F_0\bigr)(x)\int_{e>0}e^{-1}\left\{\frac{\f(x+e)-\f(x)}{\hat F_n(x+e)-\hat F_n(x)}
-\frac{\f(x)-\f(x-e)}{\hat F_n(x)-\hat F_n(x-e)}\right\}\,dF_E(e)\,dx.
\end{align*}
Let $\th_{t,\hat F_n}$ be defined by
\begin{align*}
\th_{t,\hat F_n}(e,s)=\d_1\frac{\f_{t,\hat F_n}(s)}{\hat F_n(s)}+\d_2\frac{\f_{t,\hat F_n}(s)-\f_{t,\hat F_n}(s-e)}{\hat F_n(s)-\hat F_n(s-e)}
-\d_3\frac{\f_{t,\hat F_n}(s-e)}{1-\hat F_n(s-e)}\,,
\end{align*}
where $\d_1=1_{\{s\le e\}}$, $\d_2=1_{\{e<s\le M_1\}}$ and $\d_3=1-\d_1-\d_2$. Then, for $\f=\f_{t,\hat F_n}$:
\begin{align*}
&\int_{x=0}^{M_1}(\hat F_n-F_0)(x)\int_{e=0}^{M_2} e^{-1}\left[\frac{\f(x+e)-\f(x)}{\hat F_n(x+e)-\hat F_n(x)}
-\frac{\f(x)-\f(x-e)}{\hat F_n(x)-\hat F_n(x-e)}\right]\,dF_E(e)\,dx\\
&=-\int_{e=0}^{M_2}e^{-1}\int_{x=e}^{M_1}\frac{\{F_0(x)-F_0(x-e)\}\{\f(x)-\f(x-e)\}}{\hat F_n(x)-\hat F_n(x-e)}\,dx\,dF_E(e)\\
&\qquad\qquad-\int_{e=0}^{M_2}e^{-1}\int_{x=0}^{e}\frac{F_0(x)\f(x)}{\hat F_n(x)}\,dx\,dF_E(e)\\
&\qquad\qquad\qquad\qquad\qquad\qquad+\int_{e=0}^{M_2}e^{-1}\int_{x=M_1}^{M_1+e}\frac{\{1-F_0(x-e)\}\f(x-e)}{1-\hat F_n(x-e)}\,dx\,dF_E(e)\\
&=-\int \th_{t,\hat F_n}(e,s)\,dQ_0(e,s).
\end{align*}
Note that we used, noting $\hat F_n-F_0=1-F_0-\{1-\hat F_n\}$ if $x+e>M_1$:
\begin{align*}
&\int_{e=0}^{M_2}\hat F_n(x)\int_{x=0}^{M_1} e^{-1}\left[\frac{\f(x+e)-\f(x)}{\hat F_n(x+e)-\hat F_n(x)}
-\frac{\f(x)-\f(x-e)}{\hat F_n(x)-\hat F_n(x-e)}\right]\,dF_E(e)\,dx\\
&=\int_{e=0}^{M_2}e^{-1}\int_{x=e}^{M_1}\{\f(x)-\f(x-e)\}\,dx\,dF_E(e)+\int_{e=0}^{M_2}e^{-1}\int_{x=0}^{e}\f(x)\,dx\,dF_E(e)\\
&\qquad\qquad\qquad\qquad\qquad\qquad-\int_{e=0}^{M_2}e^{-1}\int_{x=M_1}^{M_1+e}\f(x-e)\,dx\,dF_E(e)\\
&=0.
\end{align*}
Hence
\begin{align}
\label{fund}
\int \IK_h(t-x)\,d\bigl(\hat F_n-F_0\bigr)(x)=-\int \th_{t,\hat F_n}(e,s,\d_1,\d_2)\,dQ_0(e,s).
\end{align}

Replacing $\f$ by a piecewise constant function $\bar\f$, absolutely continuous w.r.t.\ $\hat F_n$, in the same way as is done on p.\ 290 of \cite{piet_geurt:14}, and defining
\begin{align*}
\bar\th_{t,\hat F_n}(e,s)=\d_1\frac{\bar\f_{t,\hat F_n}(s)}{\hat F_n(s)}+\d_2\frac{\bar\f_{t,\hat F_n}(s)-\bar\f_{t,\hat F_n}(s-e)}{\hat F_n(s)-\hat F_n(s-e)}
-\d_3\frac{\bar\f_{t,\hat F_n}(s-e)}{1-\hat F_n(s-e)}\,,
\end{align*}
we find
\begin{align*}
&\int\left|\th_{t,\hat F_n}(e,s,\d_1,\d_2)-\bar\th_{t,\hat F_n}(e,s,\d_1,\d_2)\right|\,dQ_0(e,s)\\\
&\lesssim\|\hat F_n-F_0\|_2\|\f-\bar\f\|_2=O_p\left(h_n^{-1}n^{-2/3}\right)=O_p\left(n^{-7/15}\right)=o_p\left(n^{-2/5}\right),
\end{align*}
using Lemma \ref{lemma:BM_part} in this appendix and the arguments on p.\ 333 of \cite{piet_geurt:14} (see in particular (11.49)).
Moreover,
\begin{align*}
\int\bar\th_{t,\hat F_n}(e,s)\,d\Q_n(e,s)=0.
\end{align*}
Thus we find:
\begin{align*}
&\int \IK(t-v)/h_n)\,d\bigl(\hat F_n-F_0\bigr)(v)\\
&=\int\bar\th_{t,\hat F_n}(e,s,\d_1,\d_2)\,d\bigl(\Q_n-Q_0\bigr)(e,s)+o_p\left(n^{-2/5}\right).
\end{align*}
Finally,
\begin{align*}
&n^{2/5}\int \IK(t-v)/h_n)\,d\bigl(\hat F_n-F_0\bigr)(v)\\
&= n^{2/5}\int\bar\th_{t,\hat F_n}\,d\bigl(\Q_n-Q_0\bigr)+o_p\left(n^{-2/5}\right)\\
&=n^{2/5}\int\th_{n,t,F_0}\,d\bigl(\Q_n-Q_0\bigr)+o_p(1),
\end{align*}
where
\begin{align*}
\th_{n,t,F_0}(s)
=\d_1\frac{\f_{n,t,F_0}(s)}{F_0(s)}+\d_2\frac{\f_{n,t,F_0}(s)-\f_{t,F_0}(s-e)}{F_0(s)-F_0(s-e)}
-\d_3\frac{\f_{n,t,F_n}(s-e)}{1-F_0(s-e)}\,,
\end{align*}
and $\f_{n,t,F_0}$ solve the integral equation (6.5).
The asymptotic variance $\s^2$ is therefore given by
\begin{align*}
\lim_{n\to\infty}n^{-1/5}\|\th_{n,t,F_0}\|^2_{Q_0}.
\end{align*}

The expression (6.3) for $\m$ arises from the expansion of the bias
\begin{align*}
\int \IK(t-y)\,dF_0(y)-F_0(t).
\end{align*}
\end{proof}

\subsection{Proof of (8.1)}
\label{subsection:proof_8.1}
\begin{proof}
We consider the ``mean functional'':
\begin{align*}
F\mapsto \int_{x\in[0,M_1]} x\,dF(x)
\end{align*}
The score operator (see section \ref{subsection:score_operators}) is of the form
\begin{align}
\label{score_op_single}
\left[Aa\right](e,s)=E\bigl[a(V)|(E,S)=(e,s)\bigr]=\frac{\int_{v\ge0,\,v\in(s-e,s]}a(v)\,dF(v)}{F(s)-F(s-e)}\,.
\end{align}
The adjoint is given by
\begin{align}
\label{adjoint}
\left[A^*b\right](v)=E\bigl[b(E,S)|V=v\bigr]=\int_{e>0}e^{-1}\int_{s\in(v,v+e)} b(e,s)\,ds\,dF_E(e).
\end{align}

\begin{figure}[!ht]
\centering
\includegraphics[width=0.5\textwidth]{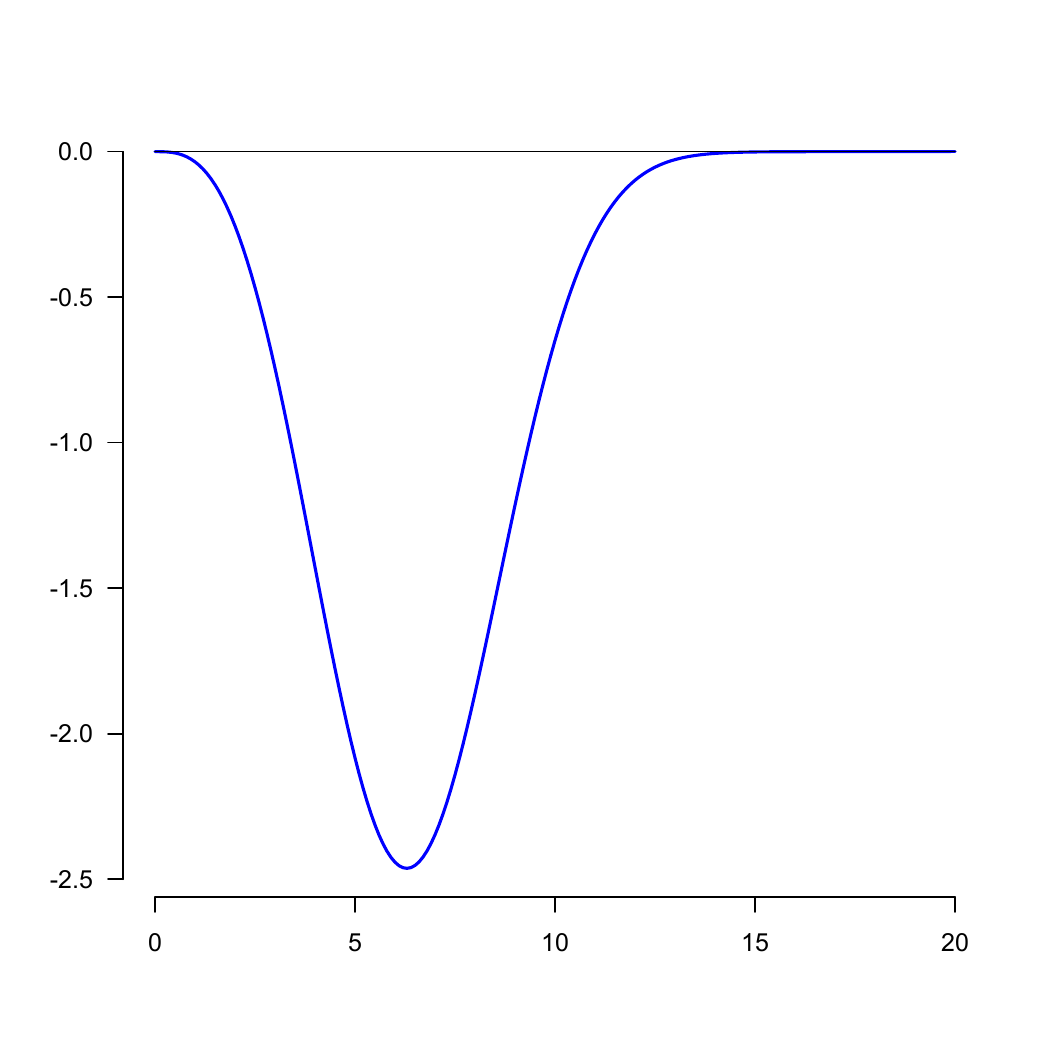}
\caption{The function $\f_{F_0}$, solving (\ref{phi-diff2}) for $F=F_0$ (Weibull).}
\label{figure:phi_mean2}
\end{figure}
 
Defining
\begin{align*}
\f_F(u)=\int_{y\le u}a(y)\,dF(y),
\end{align*}
we get the following equation for $\f_F$:
\begin{align}
\label{phi-eq2}
[A^*A\,a](v)&=\int_{e>0}e^{-1}\int_{s\in(v,v+e)}\frac{\f_F(s)-\f_F(s-e)}{F(s)-F(s-e)}\,ds\,dF_E(e)\nonumber\\
&=v-\int x\,dF(x),\qquad v\in[0,M_1].
\end{align}
By differentiating w.r.t.\ $v$, we find that $\f_F$ is also the solution of the following equation in $\f$:
\begin{align}
\label{phi-diff2}
\int_{e>0}e^{-1}\left[\frac{\f(v+e)-\f(v)}{F(v+e)-F(v)}-\frac{\f(v)-\f(v-e)}{F(v)-F(v-e)}\right]\,dF_E(e)=1,\qquad v\in[0,M_1].
\end{align}
The canonical gradient $\th_F$ is again given by:
\begin{align*}
\th_F(e,s)=\frac{\f_F(s)-\f_F(s-e)}{F(s)-F(s-e)}
=\d_1\frac{\f_F(s)}{F(s)}+\d_2\frac{\f_F(s)-\f_F(s-e)}{F(s)-F(s-e)}-\d_3\frac{\f_F(s-e)}{F(s-e)}\,,
\end{align*}
where $\d_1=\{s\le e\}$, $\d_2=\{e<s\le M_1\}$ and $\d_3=1-\d_1-\d_2$.
The solution $\f_F$ is shown in Figure \ref{figure:phi_mean2} for $F=F_0$, where we chose $F_0$ to have a Weibull distribution function.

We then get, along the lines of Chapter 10 of \cite{piet_geurt:14}, the following asymptotic normality result:
\begin{align}
\label{convergence_mean2}
\sqrt{n}\left\{\int x\,d\hat F_n(x)-\int x\,dF_0(x)\right\}\stackrel{{\cal D}}\longrightarrow N(0,\s^2),
\end{align}
where $N(0,\s^2)$ is a normal distribution with mean zero and variance
\begin{align*}
\s^2=\bigl\|\tilde\th_{F_0}\bigr\|_Q^2=-\int_0^{M_1}\f_{F_0}(x)\,dx.
\end{align*}

In fact, using $\f(M_1)=\int a(x)\,dF(x)=0$, we get:
\begin{align*}
&\bigl\|\tilde\th_{F_0}\bigr\|_Q^2=\langle A a,\tilde\th_{F_0}\rangle^2_Q
=\langle a,A^*\tilde\th_{F_0}\rangle_{F_0}=\langle a,\tilde\k_{F_0}\rangle_{F_0}=\int a(x)\tilde\k_{F_0}(x)\,dF_0(x)\\
&=\int_{x=0}^{M_1}a(x)\int_0^x\tilde\k_{F_0}'(u)\,du\,dF_0(x)=\int_{u=0}^{M_1}\tilde\k_{F_0}'(u)\left\{\int_{x=u}^{M_1}a(x)\,dF_0(x)\right\}\,du\\
&=\int_{u=0}^{M_1}\tilde\k_{F_0}'(u)\,\left\{\f_{F_0}(M_1)-\f_{F_0}(u)\right\}\,du =-\int_{u=0}^{M_1}\f_{F_0}(u)\,du.
\end{align*}
\end{proof}

\subsection{Score operators and adjoint equations}
\label{subsection:score_operators}
We need the concept of Hellinger differentiability. Let the unknown distribution $P$ on $({\cal Y,B})$ be contained in some class of probability
measures $\cal P$, which is dominated by a $\sigma$-finite measure $\mu$. Let $P$ have density $p$
with respect to $\mu$.  We are interested in estimating some real-valued function $\Theta(P)$ of $P$.

Let, for some $\d > 0$, the collection $\{P_t\}$ with $t \in (0,\d)$ be a 1-dimensional
parametric submodel which is smooth in the following sense:
$$
\int \left[ t^{-1}(\sqrt{p_t}-\sqrt{p})-
\dfrac{1}{2}a\sqrt{p} \right]^2d\mu
\rightarrow 0 \;\;\; \mbox{as } t \downarrow 0, \mbox{ for some } a \in L_2(P)
$$ Such a submodel is called {\it Hellinger differentiable}. This property can be seen as an
$L_2$ version of the pointwise differentiability of $\log p_t(x)$ at $t=0$ (with $p_0=p$),
with the function $a$ playing the role of the so-called {\it score-function} $\left.\frac{\partial}{\partial t}\log
p_t(\cdot) \right|_{t=0}$ in classical statistics. For we have, 
$$
\lim_{t \downarrow 0} \frac{\sqrt{p_t}-\sqrt{p_0}}{t}=\frac{1}{2\sqrt{p_0}} 
\left. \frac{\partial}{\partial t} p_t \right|_{t=0}=\frac{1}{2} \left( \left.
\frac{\partial}{\partial t}  \log p_t \right|_{t=0} \right) 
\sqrt{p_0}=\frac{1}{2} a
\sqrt{p_0}
$$
Therefore, $a$ is also called the {\it score function} or {\it score}.
The collection of scores $a$ obtained by considering all possible one-dimensional Hellinger differentiable parametric submodels, is a linear space, the {\it tangent space} at $P$, denoted by $T(P)$.

In the models for inverse problems, to be considered here, we work with a so-called {\it hidden space} and an {\it observation space}. All Hellinger differentiable submodels that can be
formed in the observation space, together with the corresponding score functions, are induced by the Hellinger differentiable
paths of densities on the hidden space, according to the following theorem:

\begin{theorem} \label{imagepath} Let ${\cal P} \ll \mu$ be a class of probability measures on the
hidden space $({\cal Y,B})$.  $P \in \cal P$ is induced by the random vector $Y$. Suppose that the
path $\{P_t\}$ to $P$ satisfies 
$$
\int \left[ t^{-1}(\sqrt{p_t}-\sqrt{p})-\dfrac{1}{2}a \sqrt{p} \right] ^2d\mu
\rightarrow 0
\;\;\; \mbox{as } \;t
\downarrow 0
$$ for some $a \in L_2^0(P)$, where the superscript $0$ means that $\int a\,dP=0$.\\ Let  $S:({\cal Y,B}) \rightarrow ({\cal Z,C})$ be a measurable
mapping. Suppose that the induced measures $Q_t=P_tS^{-1}$ and $Q=PS^{-1}$ on
$({\cal Z,C})$ are absolutely continuous with respect to $\mu S^{-1}$, with densities
$q_t$ and $q$. Then the path $\{Q_t\}$ is also Hellinger differentiable, satisfying
$$
\int \left[ t^{-1}(\sqrt{q_t}-\sqrt{q})-\dfrac{1}{2} \overline{a}\sqrt{q}
\right] ^2d\mu S^{-1} \rightarrow 0 \;\;\;
\mbox{as }\; t \downarrow 0
$$  with $\overline{a}(z)=E_P(a(Y)|S=z)$. 
\end{theorem}

For a proof, see \cite{bickel:98}. Note that $\overline{a} \in
L_2^0(Q)$. The relation between the scores $a$ in the hidden tangent space
$T(P)$ and the induced scores $\overline{a}$ is expressed by the mapping
\begin{align}
\label{score_operator}
A:a(\cdot) \mapsto E_P(a(Y)|S=\cdot).
\end{align}
This mapping is called the {\it score operator}. It is continuous and linear. Its range is the
induced tangent space, which is contained in $L_2^0(Q)$.

Now $\Theta:{\cal P} \rightarrow \R$ is pathwise differentiable at $P$ if for each Hellinger
differentiable path $\{P_t\}$, with corresponding score $a$, we have 
$$
\lim_{t \downarrow 0} t^{-1}(\Theta(P_t)-\Theta(P))=\Theta'_P(a),
$$
where
$$
\Theta'_P:T(P) \rightarrow \R 
$$
is continuous and linear.

$\Theta'_P$ can be written in an inner product form.   Since the tangent space $T(P)$ is a subspace of the
Hilbert-space $L_2(P)$, the continuous linear functional $\Theta'_P$ can be extended to a
continuous linear functional $\overline{\Theta}'_P$ on $L_2(P)$. By the Riesz representation
theorem, to $\overline{\Theta}'_P$ belongs a unique
$\th_P \in L_2(P)$, called the {\em gradient}, satisfying
$$ 
\overline{\Theta}'_P(h)=<\th_P,h>_P \mbox{ for all } h \in L_2(P).
$$

One gradient is playing a special role, which is
obtained by extending $T(P)$ to the Hilbert space $\overline{T(P)}$.  Then, the extension of
$\Theta'_P$ {\it is} unique, yielding the {\em canonical gradient} or {\em efficient influence
function} $\tilde{\th}_P \in
\overline{T(P)}$.  This canonical gradient is also obtained by taking the orthogonal projection of
any gradient $\th_P$, obtained after extension of $\Theta'_P$, into
$\overline{T(P)}$. Hence
$\tilde{\th}_P$ is the gradient with minimal norm among all gradients and we have (Pythagoras):
$$
\|\th_P\|_P^2=\|\tilde{\th}_P\|_P^2+\|\th_P-\tilde{\th}_P\|_P^2.
$$
  
In our censoring model, differentiability of a functional $\Theta(Q)$ along the induced Hellinger
differentiable paths in the observation space can be proved by looking at the structure of the
adjoint $A^*$ of the score operator $A$ according to theorem \ref{difth} below, which was first
proved in
\cite{VdVaart:91} in a more general setting, allowing for Banach space valued functions as
estimand. Then the proof is slightly more elaborate. 

Recall that the adjoint of a continuous linear mapping
$A: G
\rightarrow H$, with $G$ and $H$  Hilbert-spaces, is the unique continuous linear mapping $A^*:H
\rightarrow G$ satisfying 
$$ <Ag,h>_H=<g,A^*h>_G \;\; \forall g \in G, h \in H.
$$ 
The score operator from (\ref{score_operator}) is playing the role of
$A$. Its adjoint can be written as a conditional  expectation as well.  If $Z \sim PS^{-1}$, then: 
$$ [A^* \, b](y)=E_P(b(Z)|Y=y) \;\; \mbox{ a.e.-}[P]
$$
\begin{theorem}\label{difth}  Let ${\cal Q=P}S^{-1}$ be a class of probability measures on the image
space of the measurable  transformation S.  Suppose the functional $\Theta:{\cal Q} \rightarrow \R$ can be written as
$\Theta(Q)=K(P)$ with
$K$ pathwise differentiable at $P$ in the hidden space, having canonical gradient 
$\tilde{\k}$.\\ Then $\Theta$ is differentiable at $Q_P \in {\cal Q}$ along the collection of
induced paths in the observation space obtained via Theorem~\ref{imagepath} if and only if  
\begin{equation} \label{caninr}
\tilde{\k} \in {\cal R}(A^*), 
\end{equation}
where $A$ is the score operator.   If (\ref{caninr}) holds, then the canonical gradient $\tilde{\th}$ of
$\Theta$ and 
$\tilde{\k}$ of $K$ are related by 
$$
\tilde{\k}=A^*\tilde{\th}.
$$
\end{theorem}

\section{Acknowledgements}
\label{section:acknowledgements}
I want to thank the referees for their constructive remarks.

\bibliographystyle{imsart-nameyear}
\bibliography{cupbook}

\end{document}